\documentclass[reqno]{article}
\usepackage{amssymb, amsmath, amsthm, amsfonts, amscd}
\usepackage[english]{babel}

\usepackage{mathtools}
\usepackage{color}
\usepackage{hyperref}
\hypersetup{
    colorlinks=true,
    citecolor=red,
    filecolor=black,
    linkcolor=blue,
    urlcolor=black
}

\chardef\bslash=`\\ % p. 424, TeXbook
\newtheorem{theorem}{Theorem}[section]
\newtheorem{corollary}[theorem]{Corollary}
\newtheorem{lemma}[theorem]{Lemma}
\newtheorem{proposition}[theorem]{Proposition}
\newtheorem{conjecture}{Conjecture}
\newtheorem{thmx}{Theorem}

\newtheorem{corx}[thmx]{Corollary}

\newtheorem{definitionn}{Definition}[section]
\newtheorem{remark}{Remark}[section]
\newtheorem{assumption}{Assumption}[section]
\newtheorem{question}{Question}

\newcommand{\N}{\mathbb{N}}
\newcommand{\CC}{\mathcal{C}}
\newcommand{\HH}{\mathcal{H}}
\newcommand{\VV}{\mathcal{V}}
\newcommand{\DD}{\mathcal{D}}
\newcommand{\LL}{\mathcal{L}}

\newcommand{\RR}{\mathcal{R}}
\newcommand{\QQ}{\mathcal{Q}}
\newcommand{\UU}{\mathcal{U}}
\newcommand{\Z}{\mathbb{Z}}
\newcommand{\Q}{\mathbb{Q}}
\newcommand{\R}{\mathbb{R}}
\newcommand{\C}{\mathbb{C}}

\newcommand{\T}{\mathbb{T}}
\newcommand{\TT}{\mathcal{T}}

\newcommand{\AKsm}{\mathcal{AK}^{\infty}}

\newcommand{\ft}{\mathfrak{t}}

\newcommand{\SSS}{\mathcal{S}}

\newcommand{\Diff}{\mathrm{Diff}^{\infty}}
\newcommand{\Diffm}{\mathrm{Diff} ^{\infty} _{\mu} (M)}

\newcommand{\Id}{\mathrm{Id}}

\def\a{\alpha }
\def\sm{C^{\infty} }
\def\smm{C^{\infty} _{\mu} }
\def\k{\kappa }
\def\l{\lambda }
\def\L{\Lambda }

\def\r{\rho}
\def\s{\sigma}
\def\t{\tau}
\def\D{\Delta}
\def\d{\delta}

\def\w{\omega}
\def\e{\varepsilon}
\def\f{\varphi}
\def\.{\cdot }
\def\ra{\rightarrow}
\def\hra{\hookrightarrow}

\def\begeq{\begin{equation*}}
\def\endeq{\end{equation*}}

\newcommand{\vertiii}[1]{{\left\vert\kern-0.25ex\left\vert\kern-0.25ex\left\vert #1 
    \right\vert\kern-0.25ex\right\vert\kern-0.25ex\right\vert}}

\numberwithin{equation}{section}
\title{
\textsc{\textbf{Cohomological rigidity and the Anosov-Katok construction}}\\
\author{Nikolaos Karaliolios \footnote{Universit\'{e} de Lille. Email: nikolaos.karaliolios@univ-lille.fr}
}
}

\begin{document}
%%%---------------------------------------------------------------------------------------

\maketitle

\begin{abstract}
We provide a general argument for the failure of Anosov-Katok-like
constructions (as in \cite{AFKo2015} and \cite{NKInvDist}) to produce Cohomologically
Rigid diffeomorphisms in manifolds other than tori. A $C^{\infty }$ smooth diffeomorphism $f $
of a compact manifold $M$ is Cohomologically Rigid iff the equation, known as Linear
Cohomological one,
\begin{equation*}
\psi \circ f - \psi = \varphi
\end{equation*}
admits a $C^{\infty }$ smooth solution $\psi$ for every $\varphi$ in a codimension $1$
closed subspace of $C^{\infty } (M, \mathbb{C} )$.

As an application, we show that no Cohomologically
Rigid diffeomorphisms exist in the Almost Reducibility regime for quasi-periodic
cocycles in homogeneous spaces of compact type, even though the Linear Cohomological
equation over a generic such system admits a solution for a dense subset of functions
$\varphi$. We thus confirm a conjecture by M. Herman and A. Katok in that context and
provide some insight in the mechanism obstructing the construction of counterexamples.


\end{abstract}

\tableofcontents

\section{Introduction}

\subsection{Generalities and statement of the results} \label{sec generalities & results}

Let $M$ be a compact, $\sm $-smooth oriented $d$-dimensional manifold without
boundary, furnished with the volume form $\mu$. Let us also consider
$\sm _{\mu} (M, \C )$, the space of smooth functions having
$0$ mean 	with respect to $\mu$.
We study the solvability of the linear cohomological equation over any given volume
preserving diffeomorphism $f \in \Diffm $, i.e. of the equation
\begin{equation} \label{lin cohom eq}
\psi \circ f - \psi = \f
\end{equation}
where the function $\f \in \smm (M) $ is known and the unknown is $\psi \in \smm (M)$.
This equation is central in the study of dynamical systems, as it arises naturally
(cf. \cite{KatHass}, \S$2.9$), in the construction of smooth volume forms (cf. op. cit.  \S $5.1$),
in the construction of conjugations as in K.A.M. theory (cf. op. cit. \S$15.1$), and in the study
of ergodic sums (cf. \cite{FlaFo03}).
% We will
%say that we restrict the equation to $\SSS$ when we consider only
%$\phi \in \SSS$.

Let us establish some terminology concerning eq. \ref{lin cohom eq}. A function $\f
\in \smm (M)$ is called a \textit{coboundary over} $f$ if eq. \ref{lin cohom eq}
admits a solution $\psi \in \smm (M)$. The space of coboundaries over $f$ is denoted by
$Cob ^{\infty} (f) \subset \smm (M)$. A diffeomorphism $f  \in \Diffm $ is called \textit{Cohomologically Rigid} ($CR$),
iff eq. \ref{lin cohom eq} admits a solution for every $\f \in \smm (M)$, i.e.
\begeq
f \in CR(M) \iff Cob ^{\infty}(f) \equiv \smm (M)
\endeq
The diffeomorphism $f$ is called Distributionally Uniquely Ergodic ($DUE$), iff coboundaries are dense in
$\smm (M)$, i.e.
\begeq
f \in DUE(M) \iff \overline{Cob ^{\infty}(f)}^{cl_{\infty}} = \smm (M)
\endeq

A celebrated example of $DUE$ diffeomorphisms are minimal rotations in tori. Up to date,
the only known examples of $CR $ diffeomorphisms are Diophantine rotations in tori, i.e. 
rotations that are badly approximated by periodic ones
(cf. \S \ref{section rot tori} for the precise definition).
M. Herman (cf. \cite{Herm80}\footnote{Probl\`{e}me 2, p. 813. We thank St. Marmi for this reference. We also note the
following. If $G$ is a semi-simple compact Lie group and $E_{r} : \T \ra \TT \hra G$ is a regular periodic geodesic
of topological degree $r$ taking values in the torus $\TT$, then the orbit under conjugation of a cocycle in $\T ^{d} \times G$
defined by $(\a , E_{r}(\. ))$, with $\a \in DC$, is locally closed and of finite codimension (cf. ch. 8 of \cite{NKPhD}). The
codimension can go as low as $2$ for $G=SO(3)$ and geodesics of degree $1$, but in this class of examples the
codimension cannot be equal to $1$. The dynamics actually fiber over dynamics in $\T \times \TT$, and the codimension
is (known to be) finite in the space of quasi-periodic cocycles and therefore this class does not
produce counterexamples to the question raised by M. Herman.})
and then A. Katok (cf. \cite{HurdKatConj}) have posed the following conjecture.
\begin{conjecture} \label{conj CR}
The only examples, up to smooth diffeomorphism and conjugacy, of $CR(M)$
diffeomorphisms are Diophantine translations in tori.
\end{conjecture}

Anticipating the statement of cor. \ref{cor no counter-ex}, we mention that the goal of
the present article is to show that the Anosov-Katok construction is not an appropriate
tool for constructing counter-examples to this conjecture. This result is, in fact, a positive
one in guise of a negative one, as it allows us to establish that no counter-examples
exist in certain cases, cf. thm \ref{thm no counter ex in KAM informal}.

The importance of this conjecture comes from K.A.M. theory, named after Kolmogorov, Arnol'd
and Moser. A classical theorem of Arnol'd (cf. eg. \cite{KatHass} \S $15.1$)\footnote{See
\cite{NKKAMTor} for a higher dimensional analogue.}
states that a Diophantine rotation $R_{\a} :x \ra x + \a$ on the circle $\T ^{1} = \R / \Z$, if pertrubed to a real analytic
diffeomorphism whose rotation number is $\a$, will be analytically conjugate to the unpertrubed rotation. The proof
is carried out by constructing successive conjugations that make the perturbation ever smaller and showing that the
product of conjugations converges. The very construction of conjugations consists in efficiently solving the cohomological
equation over the rotation $R_{\a}$, and the argument works precisely because Diophantine rotations
are $CR(\T ^{1})$. This theorem of local linearizability is known to be false for non-Diophantine (i.e. Liouville) rotations, see
\cite{YocAst}. K.A.M. theory studies perturbations of Diophantine rotations in different contexts, and the general conclusion
is that they tend to be rigid, i.e. persist under perturbations. Application of the K.A.M.
machinery, however, depends crucially and in a general way
on the Cohomological Rigidity of the unperturbed model. The conjecture, put informally, states that this very
powerful toolcase's application is restricted to the local study of Diophantine rotations.

On the other side of the spectrum, the Anosov-Katok method of approximation by conjugation (cf. \S
\ref{subsec AK}) works for Liouville-type
rotations, i.e. rotations that are very well approximated by periodic ones. There, K.A.M.
theory fails to apply and establish rigidity, and the Anosov-Katok method actually shows that
there is no rigidity by constructing, for example, diffeomorphisms of the disc $\{ x \in \R ^{2}, \| x \| \leq 1 \}$ that
are weakly mixing and arbitrarily close to given Liouville rotations around $0 \in \R ^{2}$ (cf. the
original paper \cite{AnKat1970}). In the same context but for perturbations of
Diophantine rotations, K.A.M.
theory concludes the persistence of invariant circles (cf. \cite{Russ02}). Weak mixing is an equidistribution property
for the orbits of a dynamical system, and it is stonger than ergodicity. Consequently, the existence of an
invariant curve is an obstruction to weak mixing.

The efficiency of the Anosov-Katok construction in producing realizations of exotic dynamics makes it a good
candidate for producing counter-examples to the conjecture. To our best knowledge, there have been two recent
such attempts, both of them in spaces of quasi-periodic
skew-product diffeomorphism spaces, \cite{AFKo2015} and \cite{NKInvDist}. Both attempts
fail for slightly different reasons, but in the present article we will establish the
reason why such attempts should not be expected to produce counter-examples to the
conjecture. The reason is that the respective constructions share a key ingredient,
the Anosov-Katok method.

Before coming to the two articles cited here above, let us quickly establish some notation, which
we will also use in \S \ref{sec AR}. The notation concerns the
space of skew-product diffeomorphisms $SW^{\infty}(\T ^{d} , P)$, where $\T = \R ^{d} / \Z ^{d}$
and either $P=G $ is a compact Lie group or a homogeneous space $P=G / H$, where $H$ is a closed
subrgoup of $G$. If we let $\a \in \T ^{d}$ be a translation and $A(\. ):
\T ^{d} \ra G $ be $\sm $-smooth, then the element $(\a ,A(\. )) \in SW^{\infty}(\T ^{d} , P)$
acts on $\T ^{d} \times P$ by
\begeq
(\a ,A(\. ))(x, s.H) \mapsto (x +\a ,A(x).s.H) , \forall (x,s) \in \T ^{d} \times G
\endeq
Such an action is called a \textit{cocycle over} $\a$, or simply a \textit{cocycle}.
The space of cocycles over a fixed rotation $\a$ is denoted by $SW_{\a}^{\infty}(\T ^{d} ,P)$ and
$\a$ is called the \textit{frequency} of the cocycle.
We obviously have $SW^{\infty}(\T ^{d} , P) \subset \mathrm{Diff}^{\infty }_{\mu} (\T ^{d} \times P)$ where $\mu$ is the
product of the Lebesgue-Haar measures in $\T ^{d}$ and $P$.

Conjugation within the class of cocycles is given by fibered conjugation. Two cocycles over
the same rotation, $(\a ,A_{i}(\. ))$ for $i=1,2$, are conjugate iff there exists a
$\sm $-smooth mapping $B(\.  ):\T ^{d} \ra G $ such that
\begeq
(\a ,A_{1}(\. )) = (\a ,B(\. +\a ). A_{2}(\. ). B^{-1}(\.))
\endeq
This corresponds to a change of variables in the phase space $\T ^{d} \times P$ following
\begeq
(x,s . H) \mapsto (x, B(x ).s.H)
\endeq

In \cite{AFKo2015}, the authors worked in the space of skew-product diffeomorphisms of
$M = \T \times P$, where $\T = \R / \Z $ is the one-dimensional torus and $P$ is either a
compact nil-manifold or a homogeneous space of compact type.

They established
genericity of $DUE(M )$ in $\AKsm (\CC )$, the closure of the conjugacy class of cocycles
that are periodic diffeomorphisms of $\T \times P$:
\begeq
\CC =
\{
(p/q , A (\. )) , p,q \in \Z ^{*} , A(\. ) \in \sm ( \T , P) ,  A (\. + (q-1) \frac{p}{q}) \cdots A (\. ) \equiv \Id
\}
\endeq
Note that if $(p/q , A (\. ))$ is as above, then $(p/q , A (\. ))^{q} \equiv \Id \in 
\mathrm{Diff}^{\infty }_{\mu}  (\T ^{d} \times P)$, i.e. $(p/q , A (\. ))$ is a periodic
diffeomorphism of period $q$.

The space $\AKsm (\CC )$ was named after the Anosov-Katok construction for the exact
reason that the proof was based on periodic approximation for the frequency (i.e. the
translation acting on $\T$), and on approximation by periodic diffeomorphisms in the
fibres, i.e. $P$. For these reasons, the authors established genericity of $DUE$ in
$\AKsm (\CC )$, but were not able to address Cohomological Rigidity: Liouville numbers are
generic in $\T$, and they are those expected to be produced by a periodic approximation
argument.

In \cite{NKInvDist}, we studied a non-generic slice of the space of quasi-periodic
skew-product diffeomorphisms of $\T  \times SU(2)$, where the frequency was fixed and
satisfied a condition called Recurrent Diophantine, slightly stricter than a classical
Diophantine one. The result obtained was that $DUE$ is generic even within this non-generic
slice. Additionally, the techniques involved in the proof were precise enough so that
we were able to establish the non-existence
of $CR$ diffeomorphisms in that space. The technique of the proof, even though more precise than that of \cite{AFKo2015},
shares a basic ingredient with the latter: the Anosov-Katok argument in the fibers.

The goal of the present article is to show that the construction of counter-examples to the conjecture,
if any such counter-examples exist, is beyond present understanding of the Anosov-Katok construction.
An informal statement of the main result of the paper, to be made precise by thm
\ref{thm inst AK precise}, is the following.

\begin{thmx} \label{thm inst AK}
Let $\CC $ be a space of diffeomorphisms for which there exists $\s \in \R$ such that
for every $f \in \CC$,
\begeq
\overline{Cob^{\infty}(f)}^{cl_{\s }} \subsetneq C^{\s}_{\mu } (M)
\endeq
Then, a generic diffeomorphism in $\AKsm (\CC ) \subset \Diffm $ is not
Cohomologically Rigid.
\end{thmx}
The space $\AKsm (\CC )$ is defined in \S \ref{subsec AK} as the closure of the conjugacy
class of a class of diffeomorphisms
$\CC$. Theorem \ref{thm pert DUE} provides a similar but less precise statement on
the meagreness of $CR$ in a space where $DUE \setminus CR$ is dense (plus a technical
but important assumption). By $cl _{\s}$, we
denote the closure of a space in the $C^{\s }$ topology.

From the proof of thm. \ref{thm inst AK} and its proof, we also obtain the following corollary.
It is a statement on the Anosov-Katok method,
which is an object that does not admit an unambiguous definition. In \S
\ref{section proof cor}, and more precisely in propositions \ref{example prop rot vec 1},
\ref{example prop rot vec 2}, \ref{example prop rot vec 2 const}
and \ref{example prop comp gr}, we quantify what we mean by "Anosov-Katok-like construction" in the
statement of the corollary, thus making it a meaningful mathematical proposition.
An informal statement of the corollary is the following.
\begin{corx} \label{cor no counter-ex}
Counter-examples to the Herman-Katok conjecture, if they exist, cannot be obtained by an
Anosov-Katok-like construction.
\end{corx}
The reason why such constructions fail to produce $CR$ objects is that fast approximation
is, to a certain extent, incompatible with Cohomological Rigidity.
The precise statements provided in the proof cover with some
margin the known constructions, and the proof is structured so that all assumptions are
explicitely stated and introduced when they become relevant in the argument.
We hope that treating cases where the estimates differ slightly from our assumptions
will be facilitated this way.

Corollary \ref{cor no counter-ex} and its proof seem to indicate that the conjecture
\ref{conj CR} is true. Even though it is not known, as L. Flaminio pointed out to us,
whether Cohomological Rigidity implies the vanishing of all Lyapunov Exponents, the
failure of the most powerful method in elliptic dynamics to produce counter-examples
(unless a new arsenal of examples, allowing considerably more efficient Anosov-Katok constructions,
is discovered) suggests quite strongly that the conjecture be true.

In particular, cor. \ref{cor no counter-ex} and its proof allow us to verify the conjecture in the following setting.
\begin{thmx} \label{thm no counter ex in KAM informal}
Given $P$ a homogeneous space of compact type and $\a $ Diophantine rotation, there exists
an open set of cocycles in $SW_{\a}^{\infty}(\T ^{d} ,P)$ where $DUE \setminus CR$ is
generic but no Cohomologically Rigid cocycles exist.
\end{thmx}

Genericity of $DUE$ is of course provided by \cite{AFKo2015} and \cite{NKInvDist},
and we also prove inexistence of $CR$ cocycles. The theorem is made more precise
in \S \ref{sec AR} by thm \ref{thm no counter ex in KAM regime}.

Combining the above theorem with the so-called renormalization scheme, \cite{Krik2001} and
\cite{NKPhD}, we obtain the following corollary, valid for cocycles in $SW_{\a}^{\infty}(\T ,P)$ ($d=1$) and whose rotation satisfies a Recurrent Diophantine Condition.
\begin{corx} \label{cor no counter-ex in SW}
Given $P$ a homogeneous space of compact type and $\a $ Recurrent Diophantine rotation,
$DUE \setminus CR$ is generic in $SW_{\a}^{\infty}(\T  ,P)$,
but no Cohomologically Rigid cocycles exist.
\end{corx}

Recent advances in non-standard K.A.M. techniques (cf. \cite{AFK2011}) suggest
that the arithmetic condition can be relaxed to a classical Diophantine one.
The corollary would then hold true in
\begeq
SW^{\infty}(\T  ,P) = \bigcup _{\a \in \T} SW_{\a}^{\infty}(\T ,P)
\endeq

The results obtained herein make the following conjecture, merely a weak form of conjecture
\ref{conj CR}, seem accessible by building on the techniques developed in the present
article.
\begin{conjecture}
Let $Per_{\mu }(M) \subset \Diffm$ be the class of periodic diffeomorphisms of $M$. Then,
\begeq
\overline{Per_{\mu }(M)}^{cl_{\infty}} \bigcap CR(M) \neq \emptyset \implies M \approx \T ^{d}
\endeq
\end{conjecture}
A diffeomorphism $h \in \Diffm$ is periodic iff there exists $q \in \N ^{*}$ such that
$h^{q} \equiv \Id$.

This conjecture is obviously conjecture \ref{conj CR} restricted in the quasi-periodic
setting, which is precisely the latter's natural habitat. The most interesting case would
be that conjecture \ref{conj CR} be false, while its restriction in the quasi-periodic
setting be true. If this were so, we would be missing an important class of examples in
the theory of Dynamical Systems.

We close this section by remarking that in the statement of the conjecture here above there
is no reference to the topology of the manifold $M$. In some sense, this is in contrast
with the proof of its validity for $3$-dimensional flows (\cite{Koc09},\cite{Forn08},
\cite{Mats09}), where a case-by-case argument with respect to topology is used. In our
restriction of the conjecture, assumptions on the topology of $M$ are in fact built in the
statement, for, unless the topology of $M$ allows for interesting dynamics in
$\overline{Per_{\mu }(M)}^{cl_{\infty}}$, the rectricted conjecture is practically void.

\subsection{The Anosov-Katok method} \label{subsec AK}

A general description of the Anosov-Katok method (see \cite{AnKat1970}, \cite{FayadKatok2004}, \cite{NKInvDist})
for constructing realizations of wild dynamical
behaviours is the following. One defines a class of diffeomorphisms $\CC $, a subset
of $\Diffm$, each of whose members preserves	 a rich structure (invariant manifolds, measures,
distributions) or even the class of periodic diffeomorphisms, and whose dynamics are
quite explicit. Moreover, for the construction to produce something non trivial, one
needs that
\begeq
\CC \subsetneq \overline{\CC}
\endeq
and that diffeomorphisms in $\overline{\CC} \setminus \CC$ preserve less structure than
those in $\CC$. This depends, naturally, very heavily on the specifics of each
construction. 

One then considers the conjugacy class $\TT $ of such diffeomorphisms (where
conjugacy is in the right regularity class, usually $\sm $, and of the correct type, i.e.
volume preserving, fibered, etc.):
\begeq
\TT = \{
g \circ \tilde{f}\circ g^{-1} , \tilde{f } \in \CC , g \in \Diff (M)
\}
\endeq
So as to avoid redundancies, we assume that only one representative of each class in $\TT$
exists in $\CC$.
\begin{assumption}
If $\CC \subset \Diffm$ is a class of diffeomorphisms upon which an Anosov-Katok
method is constructed, then, for $f\in \CC $ and $g \in \Diff (M)$,
\begeq
g \circ f \circ g ^{-1} \in \CC \iff  g \circ f \circ g ^{-1} = f
\endeq
\end{assumption}

Naturally, we do not want our class $\CC$ to be closed since this assumption would remove
flexibility from the construction. It is advisable, however, to have some weaker
hypothesis of the same nature, concerning conjugacy classes. The following assumption
is a kind of precompactness for the embedding $\CC \hra \TT $.
\begin{assumption}
If $\CC \subset \Diffm$ is a class of diffeomorphisms upon which an Anosov-Katok
method is constructed, and a sequence $\tilde{f}_{n} \in \CC$ converges to $f \in \TT $,
then $f \in \CC$.
\end{assumption}

Subsequently, one looks
for realizations of the sought after behaviour in
\begeq
\AKsm (\CC ) = \overline{\TT}^{cl _{\infty}}
\endeq

To this end, a sequence $\tilde{f}_{n} \in \CC$ and
a sequence of conjugations $H_{n}$ are constructed so that the representatives
$f_{n} = H_{n} \circ \tilde{f}_{n} \circ H_{n}^{-1} \in \TT $ satisfy
\begeq
f_{n} \ra f \in \AKsm (\CC ) \setminus \TT
\endeq
The conjugations $H_{n}$ are constructed iteratively,
\begeq
H_{n} = H_{n-1}\circ h_{n}  \text{ and } H_{0} = \Id
\endeq
and $h_{n} $ is chosen so that
\begin{equation} \label{eq constr conj AK}
h_{n} \circ \tilde{f}_{n-1} = \tilde{f}_{n-1} \circ h_{n}
\end{equation}
The representative at the next step of the construction is then defined by
\begeq
f_{n} = H_{n} \circ \tilde{f}_{n} \circ H_{n}^{-1}
\endeq
where $\tilde{f}_{n}$ has to be very close to $\tilde{f}_{n-1}$ so that the $f_{n}$
converge despite the divergence of $H_{n}$.

The commutation relation implies directly that
\begin{equation} \label{eq increment AK}
f_{n}\circ f_{n-1}^{-1} = H_{n} \circ \tilde{f}_{n}\circ \tilde{f}_{n-1}^{-1} \circ H_{n}^{-1}
\end{equation}

Informally, the diffeomoprhism $\tilde{f }_{n} \circ \tilde{f}_{n-1}$ is constructed in a
scale finer than the one where $\tilde{f }_{n-1} \circ \tilde{f}_{n-2}$ was constructed, and the condition
in eq. \ref{eq constr conj AK} assures that the constructions in the respective different
scales are independent.

Since omitting a finite number of steps of the construction does not change the
asymptotic properties of the limit object $f$, we immediately get the following consequence.
If realizations of a behaviour can be constructed in $\AKsm (\CC )$, then such realizations
exist arbitrarily close to the class $\CC$ in the $\sm$ topology. This is related to the concept
of Almost Reducibility, cf. \S \ref{sec AR}.

Theorem \ref{thm inst AK precise} imposes a rate of convergence of the approximant
diffeomorphisms to the limit object in order to exclude Cohomological Rigidity.
This type of fast rate of convergence is what makes in general
the above construction work, and inasmuch as such a condition has to be built into the construction, the latter
should be expected not to produce $CR$ diffeomorphisms.

\subsection{The proofs in a nutshell}

The proof of thm. \ref{thm inst AK} says that, given a class $\CC$ of diffeomorphisms
whose coboundary space has codimension at least $1$ in $\smm (M)$, the elements $f \in \AKsm (\CC )$
for which the approximation $\TT \ni f_{n} \ra f$ is fast will not be $CR$. The speed is
measured with respect to the failure to solving the cohomological equation over $f_{n}$
for functions that oscillate slowly (the low modes of a given Laplacian on $M$).
The strength of these obstructions is measured by comparing their speed of oscillation
with their distance from $Cob^{\infty}(f_{n})$.
The fast approximation condition is proved to be generic.

The rate of approximation required so that thm. \ref{thm inst AK} be true is fast, i.e.
exponential with respect to the strength of the obstructions, which makes it Liouville-like.
The proof of cor. \ref{cor no counter-ex} focuses on the $f \in \AKsm (\CC )$ for which this
rate fails, and becomes Diophantine-like. Then, under reasonable assumptions on the class
$\CC$ and on the size of admissible conjugations, the diffeomorphisms that are
approximated at a polynomial rate will still not be $CR$.

The proof of thm. \ref{thm no counter ex in KAM informal} is based on the fact that we can
identify a class $\CC $ for which the open set of cocycles of the statement is contained
in $\AKsm (\CC )$. This class is that of resonant cocycles (cf. \S \ref{sec facts Lie groups} for the
definition), the important fact being that their analysis is very efficient and that their
coboundary space is of large codimension. In \cite{NKRigidity}, we established that a sharply polynomial rate of
approximation (i.e. polynomial and not exponential) implies the existence of a smooth
invariant foliation into tori, an obvious obstruction to $DUE$. Therefore, no $CR$ examples
exist in the corresponding $\AKsm $ space.

\textbf{Acknowledgments}: This work was partly supported by a Capes/PNPD
scholarship while the author was a post-doctoral researcher at UFF, Niter\'{o}i, Brasil,
and by the ERC AdG grant no 339523 RGDD while the author was a post-doctoral researcher
at Imperial College London, and by the LABEX CEMPI (ANR-11-LABX-0007-01).

The author is grateful to A. Kocsard for his invaluable help in preparing this work and
for many discussions on the Cohomology of Dynamical Systems, and to L. Simonelli and S.
van Strien for reading early versions of the paper and giving important feedback.

\section{Definitions, notation and preliminaries}

\subsection{General notation and calculus}

By $M$ we will denote a $\sm $ compact oriented manifold without boundary, and
by $\smm (M)$ the space of smooth (complex-valued) functions $\f$
on $M$ such that $\int _{M} \f d\mu = 0$, where $\mu $ will denote a
fixed smooth probability measure equivalent to Lebesgue, i.e. a volume form.

We will denote by $\| \. \| _{C^{s}} $ the standard $C^{s}$ norms of mappings
$M \ra E$, $E$ a normed vector space,
\begeq
\| f \| _{C^{s}} = \max _{
\substack{
x \in M \\
0\leq \s \leq s
}}
\| \partial ^{\s } f \| _{L^{\infty}}
\endeq
We will use the same notation for the countable family of semi-norms or
semi-metrics defining the topology in $\smm (M)$ and $\Diffm $. These are
\begeq
d_{s} ( \psi , \psi ' ) = \max _{0 \leq \s \leq s} \| \partial ^{\s} \psi -
\partial ^{\s} \psi ' \|
_{L^{\infty}}
\endeq
for functions, and
\begeq
d_{s} ( f _{1}, f _{2} ) = \max \{ d_{s } (f_{1} \circ f_{2}^{-1} ,\Id ) , d_{s }
(f_{2} \circ f_{1}^{-1} ,\Id ) \}
\endeq
for diffeomorphisms.

We will use the inequalities concerning the composition of functions with
mappings (see \cite{NKPhD} or \cite{KrikAst}). Here, $\psi \in C ^{\infty} (E) $
and $f$, $f_{1}$ and $f_{2}$ are smooth mappings $M \ra E$, where $E$ is a normed vector
space.
\begin{eqnarray} \label{composition}
\| \psi \circ f \| _{C^{s}} &\leq & C_{s} \| \psi \| _{C^{s}} \vertiii{f}  _{s} 
\\  \label{composition 2}
\left\Vert \psi \circ f_{2}-\psi \circ f_{1}\right\Vert _{C^{s}} &\leq &
C_{s}\left\Vert \psi \right\Vert _{C^{s+1}}\vertiii{f_{1}}  _{s} \left\Vert f_{2} \circ f_{1}^{-1}\right\Vert _{C^{s}}
\end{eqnarray}
where
\begeq
\vertiii{f} _{s} = C_{s} (1+\left\Vert f\right\Vert
_{C^{0}})^{s}(1+\left\Vert f\right\Vert _{C^{s}})
\endeq
When $f_{1} \equiv \Id $ the second inequality reads simply
\begin{equation*}
\left\Vert \psi \circ f_{2}-\psi \right\Vert _{C^{s}}\leq
C_{s}\left\Vert \psi \right\Vert _{C^{s+1}}\left\Vert f_{2} \right\Vert _{C^{s}}
\end{equation*}
For mappings $M \ra M$ and functions $\psi : M \ra \C$, this inequality stays true as long as we admit an apriori
bound on $\| f_{2} \|_{C^{0}}$ (proof by fixing a system of charts such that the ball
of a fixed radius $\d >0$ around each point is contained in a chart). The constants
would then depend on the a priori bound.

If $g $ is a fixed Riemannian metric on $M$,
inducing the measure $\mu $, then we have a natural basis for the space
$\smm (M)$. The eigenfunctions of $\Delta _{g}$, the Laplace-Beltrami operator
associated to $g$, $\{ \phi _{i} \}_{i=0}^{\infty}$ are the functions satisfying
\begin{eqnarray*}
\Delta _{g} \phi _{i} &=& - \lambda ^{2}_{i} \phi _{i} \\
%\| \phi _{i} \| _{L^{2}(\nu )} &=& 1 \\
\langle \phi _{j} (\. ) , \phi _{i} (\. ) \rangle _{L^{2}(\mu)} &=& \delta _{ij} \\
0 = \lambda _{0} &<& \lambda _{1} \leq \lambda _{2} \leq \cdots \leq \lambda _{i} \leq \cdots <\infty \\
\l _{i} &\ra & \infty
\end{eqnarray*}
%We will use the notation $k_{i} = \sqrt{\l _{i}}$.
Moreover, the sum
\begeq
\psi (\. ) = \sum _{i=1}^{\infty}\hat{ \psi} _{i} \phi _{i}(\. )
\endeq
defines a $\sm $-smooth function if, and only if,
\begeq
\hat{ \psi} _{i} = O( \l  _{i}^{-\infty})
\endeq
and every $\sm $ function admits such a representation which is unique, with
\begeq
\hat{ \psi} _{i} = \langle \psi (\. ) , \phi _{i} (\. ) \rangle _{L^{2}(\mu)}
\endeq
The functions $\phi _{i}$ satisfy the following estimate on the growth of
derivatives
\begeq
\| \phi _{i} \| _{C^{s}} \leq C _{s} \l _{i}^{s + d/2}
\endeq
see \cite{PalaisAtiyahIndThm} or \cite{Kuksin2000AnHamPDES}.
%ch. 11, thm 14.).
We also define the Sobolev spaces $H^{s} \equiv H^{s}_{g}$ for $s \in \R$,
where we drop the reference to the fixed metric $g$, by
\begeq
\{ \psi \in L^{2}(\mu ) , \sum _{i \in \N } (1 + \l _{i})^{2s} |\hat{ \psi} _{i}|^{2}
< \infty \}
\endeq
and, as usual, define the Sobolev norm in $H^{s}$ as the square root of the
sum in the definition,
\begeq
\| \psi \| _{H^{s}}^{2} =\| \psi \| _{s}^{2} = \sum _{i \in \N } (1 + \l _{i})^{2s} |\hat{ \psi} _{i}|^{2}
\endeq
and the inner product giving rise to the norm
\begeq
\langle \psi , \psi ' \rangle _{H^{s}} = \sum _{i \in \N } (1 + \l _{i})^{2s}
\hat{ \psi} _{i}\overline{\hat{ \psi} _{i}^{\prime}}
\endeq
The space $H^{-s}$ is the dual of $H^{s}$, but the only self-dual space in the
classical chain of inclusions
\begeq
\sm \equiv H^{\infty} \subset \cdots \subset H^{s_{1}} \subset \cdots \subset
H^{s_{2}} \subset \cdots
\subset H^{-\infty}  \equiv \DD '
\endeq
where $-\infty < s_{1} < s_{2} < \infty$, is $L^{2} \equiv H^{0} $. In fact, if we fix
$g$ and $\{ \phi _{i} \}$, the duality between $H^{-s}$ and $H^{s}$ is given by
\begin{eqnarray}
\langle u , \psi  \rangle _{H^{-s} , H^{s}} &=&
\langle \sum _{i \in \N } (1 + \l _{i})^{-s } \hat{u}_{i} \phi _{i } (\. ),
\sum _{i \in \N } (1 + \l _{i})^{s } \hat{ \psi} _{i} \phi _{i } (\. ) \rangle _{L^{2}} \\
&=&
\sum _{i \in \N }\hat{u}_{i} \bar{\hat{ \psi}} _{i}
\end{eqnarray}
where $u = \sum \nolimits _{i \in \N } \hat{u}_{i} \phi _{i } (\. )$ and
$\psi = \sum \nolimits _{i \in \N } \hat{ \psi}_{i} \phi _{i } (\. )$.
%This fact, already known to
%G. Orwell, is summed up by the phrase "all Hilbert spaces are naturally
%isomorphic to their dual, but $L^{2}$ is more naturally isomorpig to its dual".

For $s>0$, we will need the regularisation operators $T_{N} , \dot{T}_{N }$ and $ R_{N}$ defined by
\begin{eqnarray*}
T_{N} \psi &=& \sum _{ i \leq N } \hat{ \psi}_{i} \phi _{i } (\. ) \\
\dot{T}_{N} \psi &=& \sum _{ 0 < i \leq N } \hat{ \psi}_{i} \phi _{i } (\. ) \\
R_{N} \psi &=& \sum _{  i > N }\hat{ \psi}_{i} \phi _{i } (\. )
\end{eqnarray*}
The operators $T_{N} $ and $\dot{T}_{N}$ coincide when restricted to $\smm (M)$ or
$H_{\mu }^{s} \equiv \{ \psi \in H^{s} , \int \psi d \mu  = 0 \}$. These operators satisfy the estimates
\begin{eqnarray} \label{truncation est}
\left\Vert T_{N}\psi (\. )\right\Vert _{C^{s}} &\leq
&C_{s}\l _{N}^{s+d/2} \left\Vert \psi (\. ) \right\Vert _{C^{0}} \\
\left\Vert R_{N}\psi (\. )\right\Vert _{C^{s}} &\leq &C_{s,s'} \l _{N}^{s-s^{\prime }+d} \left\Vert
\psi (\. )\right\Vert _{C^{s^{\prime }}}
\end{eqnarray}

Since we consider a fixed volume form on $M$, namely $\mu$, we will also need the homogeneous
Sobolev spaces
\begeq
\dot{H}^{s} = \{ u \in H^{s}, \hat{u}_{0} = 0 \}, s \in \R
\endeq
We will conserve the notation $\dot{H}^{s}$ for distributions, i.e. for $s<0$, and
the notation $H_{\mu }^{s}$ for functions, i.e. for $s\geq 0$, even though there is an overlap.

The introduction of the metric $g$ serves only for providing a basis of $L^{2}$ consisting of smooth functions
and whose growth of $C^{s}$ norms satisfies the above useful properties. Alternatively, it
can be interpreted as a ruler for measuring the characteristic scale at the successive steps
of the Anosov-Katok construction by comparing them with $\l _{k}^{-1}$. We henceforth
ommit any refernce to the metric $g$, and clear it from the established notation.

For Banach spaces $E,F$, we will denote by $\LL (E,F)$ the
space of continuous linear mappings $E\ra F$ and by
$\| \. \|_{\LL(E,F)}$ the operator norm: for $T \in \LL (E,F)$,
\begeq
\| T \|_{\LL(E,F)} = \sup _{u \in E \setminus \{ 0\}}
\frac{\| T u\|}{\| u\|}
\endeq
By convention, the norm will be infinite if the operator $T$
is not continuous.

\subsection{Diffeomorphisms, cocycles and cohomology}

For this section, see \cite{Koc09}.
By $\Diff (M)$ we will denote the space of $\sm $ diffeomorphisms on $M$, and
by $\Diffm$ those that preserve the measure $\mu$. If $\UU $ is a
subspace of $\smm (M)$, and $\s \in \N$, we will denote by
\begeq
\overline{\UU}^{cl_{\s}} \subset H^{\s}_{\mu } (M)
\endeq
the closure of the space $\UU $ in the $H^{\s }$ topology.

The group $\Diffm $ acts on $\smm $ by composition: for $f \in \Diffm $ and
$\f \in \smm (M)$,
\begeq
f^{*}\f = \f \circ f
\endeq
%We will denote by $\Invf$ the set of non-$0$, $\sm$ smooth functions invariant under
%$f$:
%\begeq
%\Invf = \{ \f \in \sm \setminus \{ 0\} , f^{*}\f = \f \}
%\endeq

For every $\f \in \sm (M)$ and $f \in \Diff (M)$, we can define the real $\Z$ cocycle
in $M$ over $f$ by
\begeq
\Phi _{\f ,f} : (x,n)  \mapsto  \sum _{i=0}^{n-1} \f \circ f ^{i}
\endeq
A cocycle $\Phi _{\f ,f}$ is $C^{s}$-cohomologous to $\Phi _{\f ',f}$
iff there exists $\psi \in C^{s}(M)$ such that
\begeq
\Phi _{\f ,f} (x,n) = \psi \circ f^{n} - \psi +\Phi _{\f ',f} (x,n)
\endeq
This is equivalent to $\psi \circ f - \psi = \f - \f '$.

We shall say that $ \Phi _{\f ,f} $ is an $H^{s}$-\textit{coboundary} iff it is
$H^{s}$-cohomologous to the null cocycle, $\Phi _{0 ,f}$, which amounts to
$\psi \in H^{s} $ satisfying eq. \ref{lin cohom eq}.
For $f \in \Diffm $ and
$0 \leq s \leq \infty$ we will denote by $Cob^{s}(f) \subset \smm $ the space of smooth
functions which are $H ^{s}$-coboundaries:
\begeq
Cob^{s} (f) = \{ \f \in \smm (M) , \exists \psi \in H^{s}_{\mu}, \psi \circ f - \psi = \f \}
\endeq
This space is obtained by considering the coboundary operator as an operator $H^{s}_{\mu} \ra H^{s}_{\mu}$
and intersecting its image with $\smm \hra H^{s}_{\mu}$.

A first obstruction to a function $\f$ being a coboundary over
$f$ is related to distributions preserved by $f$ (see, e.g. \cite{Kat01}). These are the
distributions satisfying
\begeq
\langle f_{*}u , \psi \rangle :=
\langle u , \psi \circ f \rangle  = \langle u , \psi \rangle, \forall \psi
\in \sm (M)
\endeq
\begin{definitionn}
For every $f \in \Diff (M)$, we will denote the distributions in $H^{-s} \setminus \{ 0\}$ (resp. $\dot{H}^{-s}\setminus \{ 0\}$) that
are preserved by $f$ by $H^{-s} (f)$ (resp. $\dot{H}^{-s}(f)$). By $\DD ' (f) $ we denote the distributions in
$\DD ' \equiv H^{-\infty}$ preserved by $f$. Since $M$ is compact,
\begeq
\DD ' (f) = \{ 0 \} \cup \bigcup _{s \in \R } H^{s}(f)
\endeq
Clearly, $\R \mu \subset \DD '(f)$, for all $f \in \Diffm$.
Whenever we write $u \in  \dot{H}^{-s} (f)$, we implicitly assume that $\| u \| _{H^{-\s}} =1 $.

We also denote by $\dot{\HH }^{-s}(M)$ the diffeomorphisms in $\Diffm $ that preserve a
non-trivial distribution in $\dot{H}^{-s} (M)$:
\begeq
\dot{\HH }^{-s}(M) = \{ f \in \Diffm  , \dot{H}^{-s} (f) \neq \emptyset
\}
\endeq
\end{definitionn}

It follows immediately from the definition that any $\sm$-coboundary
$\f \in Cob^{\infty}(f)$ must satisfy
\begin{equation} \label{eq in ker of ob}
\langle u , \f  \rangle = 0 ,\, \forall u \in \DD ' (f)
\end{equation}
Under this condition on $\f$, the Hahn-Banach theorem shows that it is actually
an \textit{approximate coboundary}, i.e. that for every $\e >0 $ and $s_{0} \in \N$,
there exist $\psi , \epsilon \in \smm (M)$ satisfying
\begin{eqnarray*}
\psi \circ f - \psi &=& \f + \epsilon \\
\| \epsilon \| _{C^{s_{0}}} &<& \e
\end{eqnarray*}

However, the condition of eq. \ref{eq in ker of ob} is not sufficient for a function to be a coboundary,
and the application of the Hahn-Banach theorem gives an optimal answer in full generality.
A celebrated example is that of Liouvillean rotations, for which we refer the reader to the next section,
and especially to prop. \ref{prop DC is CF}.

The following nomenclature concerning the properties of a diffeomorphism relative to the space of its
coboundaries is more or less standard.

\begin{definitionn}
A diffeomorphism $f\in \Diffm $ is called $DUE$, standing for Distributionally Uniquely Ergodic,
iff $\dim \DD ' (f) = 1$, in which case $\DD ' (f)$ is the vector space
generated by the unique invariant probability measure $\mu$:
\begeq
f \in DUE(M) \iff \DD '(f) = \R \mu
\endeq
\end{definitionn}

The following lemma follows trivially from the definition, since $DUE $ diffeomorphisms
are uniquely ergodic with respect to $\mu$.
\begin{lemma}
A $DUE(M)$ diffeomorphism is minimal.
\end{lemma}

We now introduce the following distinction with resepect to the coboundary space being
closed or not.

\begin{definitionn}
The diffeomorphism $f$ is called $CS$, Cohomologically Stable, iff the space of coboundaries
$Cob^{\infty}(f)$ is closed in the $\sm$ topology.
\begeq
f \in CS(M) \iff Cob^{\infty}(f) =  \overline{Cob^{\infty}(f)}^{cl_{\infty}}
\endeq
\end{definitionn}
\begin{definitionn}
A diffeomorphism $f\in \Diffm $ is called $CR$,
Cohomology Rigid, iff it is both $DUE$ and $CS$, i.e. iff $Cob^{\infty}(f)$ is
closed and of codimension $1$ in $\sm (M)$. We thus have
\begeq
f \in CR(M) \iff Cob^{\infty}(f) = \smm (M)
\endeq
\end{definitionn}
%Given $\SSS \subset \sm$, a closed space of observables, we will use the notation
%$DUE _{ \uhpr \SSS}$ for the set of diffeomorphisms $f$ such that $\SSS \subset \overline{Cob^{\infty}(f)}$,
%which is equivalent to $u \in H^{-s}(f) \Rightarrow \SSS \subset \ker u$. By $CS _{ \uhpr \SSS}$
%we will denote the $f \in \Diff$ satisfying $\SSS \subset Cob^{\infty}(f)$. Obviously,
%$CS _{ \uhpr \SSS} \subset DUE _{ \uhpr \SSS}$.

Finally, we recall the following lemma on the canonicality of the cohomological equation
and state an obvious, yet important in the context we are interested in, corollary.
\begin{lemma} \label{lem cohom eq ch coord}
Let $f,g,\tilde{f} \in \Diffm$ be such that $\tilde{f} = g \circ f \circ g^{-1}$. Then,
for functions $\psi ,\f , \tilde{\psi},\tilde{\f} \in \smm$,
\begeq
\psi \circ f - \psi = \f \iff \tilde{\psi} \circ \tilde{f} - \tilde{\psi} = \tilde{\f}
\endeq
where $\tilde{\psi} = \psi \circ g^{-1}$, same for $\tilde{\f}$.
\end{lemma}
The proof is by direct calculation, and the following corollary is immediate.
\begin{corollary} \label{cor com ch coord}
If $g$ commutes with $f$, then
\begeq
\psi \circ f - \psi = \f \iff \tilde{\psi} \circ f - \tilde{\psi} = \tilde{\f}
\endeq
\end{corollary}

In particular, whenever we fix a class $\CC$ serving as a basis for an Anosov-Katok
construction, the corollary will be of use, since functions in $\CC$ are chosen so that
they have many symmetries.

\subsection{Rotations in tori and arithmetics} \label{section rot tori}

A vector $\a \in \T ^{d} = \R ^{d} / \Z ^{d}$ will be called \textit{irrational} iff
\begeq
\langle q , \a \rangle \in \Z \text{ and } q \in \Z ^{d}
\Rightarrow q = 0
\endeq
The vector $\a \in \T ^{d}$ induces a minimal rotation $x \ra x +\a$ on the torus $\T ^{d}$
iff it is irrational. We can distinguish between two types of irrational vectors through
the following definitions. The justification is given just below, in prop. \ref{prop DC is CF}.

\begin{definitionn}
A vector $\a \in \T ^{d} $ is called Diophantine iff it
satisfies a Diophantine Condition of type $\tilde{\gamma} , \tilde{\t }$, for some $\tilde{\gamma} >0$
and $ \tilde{\t } > d$. Such a condition, denoted by $DC ( \tilde{\gamma} ,  \tilde{\t } )$, is defined
by
\begeq
\a \in DC ( \tilde{\gamma} ,  \tilde{\t } ) \Leftrightarrow
|\langle q , \a \rangle |_{\Z} \geq \frac{\tilde{\gamma} ^{-1}}{|q|^{ \tilde{\t }}} , \forall
q \in \Z ^{d} \setminus \{ 0 \}
\endeq
\end{definitionn}
The distance from $\Z$, $|\w |_{\Z}$, for $\w \in \R$ is defined by
\begeq
| \w | _{\Z} = \min _{p \in \Z } |\w - p|
\endeq
and the norm on $\Z ^{d}$ is the $\ell ^{1}$ norm.
It is a classical fact that Diophantine vectors form a meagre set of full Haar measure
in $\T ^{d}$.
\begin{definitionn}
Liouville vectors in $\T ^{d}$ are denoted by $\LL $ and are the irrational
vectors which do not satisfy any Diophantine Condition.
\end{definitionn}
It is an equally well established result that Liouville vectors form a residual set of $0$
measure in $\T ^{d}$.
We now recall the proof of the fact that an irrational rotation in $\a \in \T ^{d}$ is
$CR$ iff $\a \in DC$.
\begin{proposition} \label{prop DC is CF}
Let a rotation $R_{\a }: \T ^{d} \ra \T ^{d}$, $R_{\a } : x \mapsto x+\a \mod \Z ^{d}$.
Then, $R _{\a} \in CR(\T ^{d} )$ iff $\a \in DC$.
\end{proposition}
\begin{proof}
If $\a \in DC ( \tilde{\gamma} ,  \tilde{\t } )$ and $\f \in C^{\infty}_{\mu }$, then there is a unique
$\psi \in C^{\infty}_{\mu }$ solving
\begeq
\psi (\. + \a) - \psi (\. ) = \f (\. )
\endeq
and satisfying the estimate $\| \psi \|_{s} \leq C_{s} \gamma \| \f \|_{s+\tau}$. The first
step of the proof is application of the Fourier transform in order to obtain the equation
\begeq
\hat{\psi} (k) = \frac{1}{e^{2i\pi \langle k , \a \rangle } -1} \hat{\f} (k)
\endeq
We then estimate the norm using the definition of the Diophantine condition. The factor
$(e^{2i\pi \langle k , \a \rangle } -1)^{-1}$ is known as a \textit{small denominator}.

If $\a \in \LL $, let $q_{n} \in \Z ^{d}$ be such that
$|\langle q_{n} , \a \rangle |_{\Z^{d}} \leq |q _{n}|^{-n}$. Then, the function
\begeq
\f (\. ) = \sum _{n} (e^{2i\pi \langle q_{n} , \a \rangle } -1)^{1/2}
e^{2i\pi \langle q_{n} , \.   \rangle}
\endeq
is in $ C^{\infty}_{\mu } (\T ^{d})$, but the solution is not defined in any function or distribution space,
since the modulus of its $q_{n}$-nth Fourier coefficient grows faster than any power of $q_{n}$.
\end{proof}

A straightforward application of the proposition above and of the definition of Cohomological
Rigidity shows that the only $CR( \T ^{d} )$
diffeomorphisms homotopic to the $\Id $ are, up to smooth conjugation, Diophantine translations (see \cite{Koc09}
for the details).

Let us also define the Recurrent Diophantine condition. We call
$\mathrm{G}: \T \setminus \{ 0 \}\ra \T $ the Gauss map $x \mapsto \{x ^{-1} \}$, where
$\{ \. \}$ denotes the fractional part of a real number. For $\a \in \T \setminus \Q$, call
$\a _{n} = \mathrm{G}^{n} (\a )$.
\begin{definitionn} \label{def RDC}
A rotation $\a \in \T$ satisfies a Recurrent Diophantine condition of type $\tilde{\gamma} , \tilde{\t }$ iff
$\a _{n} \in DC (\tilde{\gamma} , \tilde{\t } )$ for infinitely many $n \in \N$.
\end{definitionn}

It is a full Haar measure condition for every $\tilde{\gamma} >0$ and $\tilde{\t } > 1$, and,
put informally, demands that when we apply the continued fractions algorithm on
$\a$, the remainders of the Euclidean division satisfy a fixed Diophantine condition
infinitely often.

\section{Fundamental lemmas}
In this section, we prove four fundamental lemmas. They
provide information on the behaviour of the coboundary
operator and its inverse under perturbation of the
diffeomorphism in three important settings.

The first and third lemmas can be seen as abstractions of
what happens when we perturb a rational rotation and look
for a solution to the cohomological equation for a rhs
function supported in the modes where the denominator is $0$ or non-$0$, respectively.
The second lemma concerns the same setting but corresponds to
perturbations of Liouville rotations and modes whose denominators
are Liouville-small.

\subsection{The effect of obstructions}

This first lemma, basic ingredient of the proof of thm. \ref{thm inst AK precise}, provides
an estimate which quantifies the following fact. Given a diffeomorphism
$f' \in \Diffm $ which preserves a distribution $u \in \dot{H}^{-s} (M)$ and a function
$\f \in H^{s}_{\mu}$, $\f \notin \ker u$, if we perturb $f' $ to $f$
in the $C^{s}$ topology and assume that
$\f \in \overline{Cob^{\infty}(f )}$,
then the estimates on the norms of the solution (or an approximate one) should
be expected to be bad. The following lemma provides a precise statement, and its proof
is to be compared with the small denominator estimate for irrational rotations.

\begin{lemma} \label{badness estimate}
Let $f \in \Diffm  $ and suppose that there exists $f' \in \Diffm $ such that
$\| f \circ (f')^{-1} \|_{C^{\s }} = \delta _{\s } >0$ is small and such that there exists
some $u \in \dot{H} ^{-\s } (f')$, $\| u \|_{H^{-\s}} =1$,
Suppose, now, that there exists an approximate solution $\psi \in C ^{\s +1}$ to the cohomological equation, i.e.
\begin{equation} \label{eq lin cohom S}
\psi \circ f - \psi = \f + \epsilon
\end{equation}
with $\f \in \smm \setminus \ker u $ and $\| \epsilon \| _{H^{\s} } = \e _{\s }$
small enough:
\begeq
|\langle u , \epsilon \rangle | \leq \frac{1}{2}|\langle u , \f \rangle |
\endeq
Then, $\psi $ satisfies the following estimate
\begeq
\| \psi \|_{C^{\s +1}} \geq C_{\s } \vertiii{f} _{\s} \dfrac{|\langle u , \f \rangle |}{\delta_{\s }}
\endeq
\end{lemma}

\begin{proof}
The proof uses the estimates for composition of mappings and the
invariance of the objects. Eq. \ref{eq lin cohom S} and the fact that $(f')^{*} u = u$
imply that
\begeq
\langle u ,
 \psi \circ ( f \circ (f') ^{-1}) -  \psi \rangle =
 \langle u , \f  + \epsilon  \rangle
%\endeq
%\psi \circ ( f' \circ f ^{-1}) -  \psi = \phi \circ f ^{-1} +
%\psi \circ f ^{-1} -  \psi + \epsilon \circ f ^{-1}
\endeq
%Then, the fact that $f^{*} u = u$ implies that
%\begin{eqnarray*}
%|\langle u , \psi \circ ( f' \circ f ^{-1}) -  \psi \rangle | &=&
%| \langle u , \phi \circ f ^{-1} 
% + \epsilon \circ f ^{-1} \rangle | \\
% &=&
%| \langle u , \phi   + \epsilon  \rangle |
%\end{eqnarray*}
Estimation by duality, the assumed smallness of $|\langle u , \epsilon \rangle |$
and the triangle inequality imply directly that
\begeq
\| u \| _{H^{-\s }}
\| \psi \circ ( f \circ (f') ^{-1}) -  \psi \| _{H^{\s }}
\geq \frac{1}{2}| \langle u , \f \rangle |
\endeq
The estimate announced in the statement of the lemma follows from the
inequality on the composition of functions with mappings, eq.
\ref{composition} and \ref{composition 2}.
\end{proof}

\subsection{The effect of cohomological instability}

This second lemma is more qualitative in its nature. It is used in the proof of thm.
\ref{thm pert DUE}, which is consequently less precise than thm. \ref{thm inst AK precise}.

\begin{lemma} \label{badness estimate 0}
Let $f \in \Diffm$, and suppose that it is not cohomologically stable, i.e. that
\begeq
Cob^{\infty}(f)\subsetneq \overline{Cob^{\infty}(f)}^{cl_{\infty}}
\endeq
and fix $\f \in \overline{Cob^{\infty}(f)}^{cl_{\infty}} \setminus Cob^{\infty}(f)$,
an approximate but not exact coboundary over $f$.

Fix some $\d >0 $ and a $s _{0} \in \N$ big enough. Then, for every $M>0 $, there exists $\e >0 $
such that, for every $s_{1} \geq s_{0}+1$, if we call $ \d _{s_{1}} =d_{s_{1}} (f,f') <\d $, then
\begeq
\psi  \circ f' - \psi = \f
\endeq
implies that
\begeq
\begin{cases}
\| \psi \| _{s_{0}} > M \text{, or} \\
\| \psi \| _{s_{1}} > C^{-1}_{s_{1}} \e \vertiii{f} _{s_{1}} ^{-1} \d _{s_{1}} ^{-1}
\end{cases}
\endeq
In particular,
\begeq
\| \psi \| _{s_{0}+1} > \max \{ M , C^{-1}_{s_{0}} \e \vertiii{f} _{s_{0}} ^{-1} \d _{s_{0}} ^{-1} \}
\endeq
\end{lemma}
\begin{proof}
Let $f$ and $\f$ be as in the statement of the lemma. Then, by the
Ascoli-Arzel\`{a} theorem, there exists $s_{0}$ such that if $s_{1}\geq s_{0}+1$, then for every $M>0$
there exists $\e >0$ such that if $\psi \in \smm$ and
\begeq
\| \psi \circ f - \psi - \f \|_{s_{1}} < \epsilon \Rightarrow
\| \psi \|_{s_{0}} > M
\endeq

Let us fix such $M$ and $\e$, and suppose that
$d_{s} (f,f') = \d _{s}$ is small enough for some $s \geq s_{1}$. If, now,
$\psi $ is such that $\psi \circ f' - \psi = \phi$, then
\begeq
\psi \circ f - \psi = \f + \psi \circ f - \psi \circ f'
\endeq
Now, either
\begeq
\| \psi \| _{s_{0}+1} \geq C^{-1}_{s_{0}} \e \vertiii{f} _{s_{0}} ^{-1} \d _{s_{0}} ^{-1} % \geq C\epsilon ||| f |||_{s} ^{-1} \d _{s} ^{-1}
\endeq
%where $0<C<1$ depends on an apriori upper bound on $\d _{1}$ and $C\ra 1$ as this upper bound
%goes to $0$,
where $C_{s_{0}} $ is the constant appearing in eq. \ref{composition 2}, and
\begeq
\| \psi \| _{s} \ra \infty \text{ as } \d _{s_{0}} \ra 0 , \forall s \geq s_{0}+1
\endeq
or
\begeq
\| \psi \circ f - \psi \circ f' \| _{s_{0}} \leq
C_{s_{0}}\| \psi \| _{s_{0}+1} \vertiii{f} _{s_{0}} \d _{s_{0}} < \e
\endeq
In the second case, by the Cohomological Instability of $f$ we obtain that
\begeq
\| \psi \| _{s_{1}} > M
\endeq
and therefore $\| \psi \| _{s_{1}} > \max \{ M ,C^{-1}_{s_{1}} \e \vertiii{f} _{s_{1}} ^{-1} \d _{s_{1}} ^{-1} \}$
\end{proof}

\subsection{A continuity estimate}

Lemma \ref{badness estimate} provides a lower bound for
the solution of the cohomological equation in the presence
of obstructions. The estimate giving a lower bound of the same
nature but by using the coboundaries of apporoximating
diffeomorphisms is an easier statement provided by the
following lemma.

The following lemma will be exploited in showing blow-up of
solutions in the proof of corollary \ref{cor no counter-ex}.
\begin{lemma} \label{badness estimate cob}
Let $f' \in \Diffm $, $\f ' \in Cob (f') $ and $\psi '$ a
solution to the cohomological equation:
\begin{equation}
\psi \circ f' - \psi = \f '
\end{equation}

Consider, now, $f \in \Diffm $ to be $CR$ and close to $f'$,
so that
\begeq
\| f \circ (f')^{-1} \|_{C^{s }} = \delta _{s }
\endeq
is small enough for $s \leq s_{0}$ big enough.
Then, the inverse of the coboundary operator,
\begeq
Cob_{f}^{-1}: Cob(f) \equiv \smm (M) \ra \smm (M)
\endeq 
satisfies for all $s,s' \leq s_{0}$
\begeq
\| Cob_{f}^{-1}\|_{\LL(H^{s},H^{s'})} \geq
\frac{\|\psi \|_{s'}}{\|\phi '\|_{s }+
C_{s}\vertiii{f} _{s}\|\psi \|_{s +1}\d _{s}}
\endeq
\end{lemma}

%\begin{lemma} \label{goodness estimate}
%Let $f \in \Diffm  $ be $CR$, and fix an inverse of the coboundary operator,
%\begeq
%Cob_{f}^{-1}: Cob(f) \ra \smm (M)
%\endeq
%satisfying for all $\s \geq 0$ and for some function $\t : \R_{\geq 0} \ra \R_{\geq 0}$
%\begeq
%\| Cob_{f}^{-1}\|_{\LL(H^{\s},H^{\s -\t (\s)})} = C_{\s} >0
%\endeq
%
%
%Consider, now, $f' \in \Diffm $ for which
%$\| f \circ (f')^{-1} \|_{C^{\s }} = \delta _{\s } >0$ is small.
%Then, if $\f \in Cob (f') $ and $\psi '$ is a solution to the cohomological equation, i.e.
%\begin{equation}
%\psi ' \circ f' - \psi '= \f
%\end{equation}
%$\psi '$ satisfies the following estimate
%\begeq
%\| \psi \|_{C^{\s -\t +1}} \leq C_{\s } |||f|||_{\s} (1+ \d _{\s}) \| \f \| _{\s}
%\endeq
%\end{lemma}
For the lemma to become relevant, we need to have control over
the growth of the derivatives of the solution $\psi$, so that
it can be compared with the distance $\d _{s}$. If we have such
information, the lemma provides
an estimate of the kind that semi-continuous dependence of
the norms of the coboundary operator would provide, if this
property were proved to hold.

\begin{proof}
By construction, the function $\psi $ is a solution to the equation
\begin{equation}
\psi \circ f - \psi = \f = \f ' + \psi \circ f - \psi \circ f'
\end{equation}
The estimate follows by application of the inequalities on
composition, eq. \ref{composition} and \ref{composition 2}
and the definition of the operator norm. The function $\psi $ is a solution
of minimal norm in any $H^{s}$ thanks to the hypothesis that $f \in CR$.
\end{proof}

\begin{corollary}
If in the setting of the lemma \ref{badness estimate cob}
\begeq
C_{s}\vertiii{f} _{s}\|\psi \|_{s +1}\d _{s} \leq \frac{1}{2}\|\phi '\|_{s }
\endeq
then
\begeq
\| Cob_{f}^{-1}\|_{\LL(H^{s},H^{s'})} \geq \frac{2}{3}
\frac{\|\psi \|_{s'}}{\|\phi '\|_{s }}
\endeq
\end{corollary}
The implication is obvious. The $1/2$ factor in the hypothesis assures
that, in the notation of the lemma, $\f \not\equiv 0$. The corollary
(as well as the lemma) are relevant when
\begeq
\frac{\|\psi \|_{s'}}{\|\phi '\|_{s }}
\endeq
is a good estimate for $\| Cob_{f'}^{-1}\|_{\LL(H^{s},H^{s'})}$ and
$f$ is close to $f'$ with respect to the size of $\psi$.

Another way of exlpoiting continuity is given by the following lemma.
\begin{lemma} \label{lem closed graph}
Let $f_{n} \ra f \in CR(M)$ and functions $\f _{n} \in Cob^{\infty}(f_{n})$ such that if
\begeq
\psi _{n} \circ f_{n} - \psi _{n} = \f _{n} 
\endeq
with
\begeq
\| \partial \psi _{n} \|_{s} d(f,f_{n}) \ra 0 , \forall s
\endeq
then
\begeq
\psi _{n} \ra \psi
\endeq
where $\psi$ satisfies
\begeq
\psi \circ f - \psi = \f
\endeq
\end{lemma}
\begin{proof}
We calculate
\begin{equation*}
\begin{array}{r@{}l}
\psi _{n} \circ f- \psi _{n} &= \psi _{n} \circ f - \psi _{n} \circ f_{n} + \psi _{n}
\circ f_{n}- \psi _{n} \\
&= o(1) + \f _{n} \ra \f
\end{array}
\end{equation*}
The continuity of the inverse of the Coboundary operator of $f$ implies the lemma.
\end{proof}

\section{Proof of theorem \ref{thm inst AK}}
We can now state a precise version of thm \ref{thm inst AK}.
\begin{theorem} \label{thm inst AK precise}
Suppose that $\SSS \subset \Diffm $ is a closed subspace such that
$\dot{\HH }^{-\s }(M) \cap \SSS$ is dense in $\SSS$ for some $\s \geq 0$.
Then, $CR(M) \cap \SSS $ is meagre (or empty).
\end{theorem}

Before providing the proof for this theorem, we remark that if it also happens (as in
\cite{AFKo2015} and \cite{NKInvDist}) that $DUE(M) \cap \SSS$ is dense, then
$DUE$ is actually generic in $\SSS$, since for general reasons $DUE$ is a $G_{\d}$ property:
\begin{equation} \label{eq DUE is Gd}
DUE(M)=
\bigcap _{m,n,k} \{
f \in \Diffm  , \exists \psi \in \smm ,
\| \psi \circ f - \psi - \phi _{k} \|_{C^{n}} < m^{-1}
\}
\end{equation}
where $\{ \phi _{k} \}$ is the basis of eigenfunctions of the Laplacian on $M$. However, there
is no apriori topological reason why $CR(M) \cap \SSS $ should not be empty.
The initial goal of this paper was in fact to prove that $CR$ is an $F_{\s}$ property (just
as $DC$), thus establishing the difficulty of the conjecture in full generality. We still do
not know whether this is true.

Theorem \ref{thm inst AK precise} explains why the techniques of \cite{AFKo2015} fail to conclude about the existence of
a counterexample to the conjecture, since they only provide information on generic diffeomorphisms in
the space $\AKsm $ as it is defined in the reference, where $\dot{\HH }^{-\s}(\T \times P)$ is dense for every
$\s \geq 0$: a generic diffeomorphism in that space has to be $DUE$ and not $CR$.
It also explains why the hands-on approach of \cite{NKInvDist} is needed in order
to exclude the existence of $CR$ diffeomorphisms in the respective space of dynamical systems.
%
%\subsection{Preparation of the proof}
%
%
%

\subsection{A proposition on approximation}
We now prove the following proposition concerning
the instability of a diffeomorphism $f$ that is well approximated by diffeomorphisms preserving distributions.
The proposition shows that, under quite mild conditions, $f$ will not be $CR$. The proof consists in producing
a function $\f \in \smm$, $\f \not\equiv 0$, which is not a coboundary over $f$.

\begin{proposition}  \label{prop approx}
Let $f \in \Diffm $ and suppose that there exists a sequence of diffeomorphisms $f_{n} \in \Diffm $,
$f_{n} \ra f $ in $\sm$,
satisfying the following properties:
\begin{enumerate}
\item There exists $\s \geq 0$ such that, for every $n \in \N$, $f_{n } \in \dot{\HH}^{-\s }(M)$.
\item If we let $\d _{n} = \d _{\s,n} = d_{\s}(f,f_{n}) \searrow 0 $, then there exist
a sequence $u_{n} \in \dot{H}^{-\s }(f_{n}) $ with $\| u_{n} \| _{H^{-\s}}=1$, a
sequence $N_{n} \in \N$, $K>0$ and $s_{0}\geq 0$, such that
\begin{equation} \label{eq spec un}
\| T_{N_{n}} u_{n}  \|_{H^{-\s }} \geq K \l _{N_{n}} ^{-s_{0} }
\end{equation}
and
\begeq
\d _{n} = O(\l _{N_{n}}^{-\infty})
\endeq
\end{enumerate}
Then, there exists $\w \in \smm (M)$ which is not an exact coboundary over $f$: $\w \notin Cob^{\infty}(f)$.
\end{proposition}
%The condition that $N_{n}\ra \infty$ is a non-degeneracy one imposed on the space
%generated by the $u_{n}$. If they live in a finite-dimensional
%subspace of $H^{-\s}$, then by continuity, $H^{-\s }(f) \neq \emptyset$, so that $f $ is not $DUE$ and the proposition
%is irrelevant.

The condition of item $2$ of the theorem compares the rate of
convergence of the $f_{n}$ to $f$ with the spectrum of the
$u_{n}$ and asks that the former be infinitesimal with respect
to the latter (condition to be compared with the approximation
of a Liouvillean number by its best rational approximations).
We point out for later use that, up to considering a subsequence
depending on a fixed $M>0$, we can impose that
%the fast decay of $\d _{n}$ with respect to $N_{n}$ implies that
\begeq
\d _{n}^{-1} \sum _{k>n} \d _{k} < M
\endeq

Under the assumption of the existence a sequence $f_{n} \in \Diffm $ such that
$f _{n} \ra f $ fast enough and such that $H^{-\s }(f_{n}) \neq \emptyset$ for some $\s > 0$
and for every $n \in \N$, then
lemma \ref{badness estimate} becomes relevant. If we let
$\d _{s,n} = d_{s}(f,f_{n}) \searrow 0 $,
and we suppose that $\w \in \smm (M)$ and $u_{n} \in H^{-\s }(f_{n})$ is such that
\begin{equation}  \label{eq. distr. ag cob}
 \d _{\s ,n}^{-1} |\langle u_{n} , \w \rangle |\ra \infty
\end{equation}
by lemma \ref{badness estimate}, this would force a solution $\psi$ of the cohomological
equation
\begeq
\psi \circ f - \psi = \w
\endeq
to satisfy $\| \psi \|_{\s +1} = \infty$, so that no such smooth function can exist.
We stress that we do not claim that $\w \in \overline{Cob^{\infty}(f)}^{cl_{\infty}}$.
We make the following remark for future use.
\begin{remark} \label{rem depend on param}
Each truncation order $N(\tilde{f})$ is non-increasing with respect to $s_{0}$ and
non-decreasing with respect to $K$.
\end{remark}

We now construct such a function $\w$.

\begin{proof}
Let us chose $u_{n} \in \dot{H}^{-\s }(f_{n})$,
$\| u_{n} \| _{H^{-\s}} =1$, and $\w ^{(n)} \in H^{\s }_{\mu } (M)$, $\| \w^{(n)} \| _{H^{\s}} =1$ with\footnote{Here
we implicitly identify $H^{\s }$ with its dual. Since we keep $\s $ fixed, there are no
complications. The argument would become more involved if one wishes to allow $\s = \s _{n} \ra \infty$.}
\begeq
\w ^{(n)} \perp _{H^{\s}} \ker u_{n}
\endeq
Since our goal is to construct a $\smm$ function, we truncate the functions $\w ^{(n)}$
in order to obtain
\begeq
\w _{n} (\. ) = T_{N_{n}} \w ^{(n)} (\. ) \in \smm (M)
\endeq
and sum them following
\begeq
\w (\. ) = \sum _{k=1}^{\infty} c_{k} \d _{k}\w _{k} (\. )
\endeq
where $\d _{n} = \d _{\s ,n} $
The function thus defined will be smooth provided that $\d _{n} = O(\l _{N_{n}}^{-\infty})$, and
$|c_{n}| = O(\l _{N_{n}}^{s_{1}})$ for some $s_{1}<\infty$.

Under these conditions, let us calculate and estimate the lhs of the limit in eq. \ref{eq. distr. ag cob}:
\begin{equation} \label{eq. norm dist ag cob}
\d _{n}^{-1}  \langle u_{n} , \w \rangle  =
\d _{n}^{-1} \sum _{k<n} c_{k} \d _{k} \langle u_{n} , \w  _{k} \rangle
+
c_{n} \langle u_{n} ,\w _{k}  \rangle
+
\d _{n}^{-1} \sum _{k>n} c_{k} \d _{k} \langle u_{n} , \w _{k} \rangle
\end{equation}
No reasonable assumption seems to exist that imposes restrictions on the first sum.
For example, $\langle u_{n} , \w _{k} \rangle \ll \d _{n}, 0<k<n$, seems to be needlessly restrictive.
Fortunately, such an assumption appears to be unnecessary: we need only consider the sign
of the sum,
\begeq
\begin{cases}
+1, \text{ if } \sum _{k<n} c_{k} \d _{k} \langle u_{n} , \w  _{k} \rangle \geq 0 \\
-1, \text{ if } \sum _{k<n} c_{k} \d _{k} \langle u_{n} , \w  _{k} \rangle < 0
\end{cases}
\endeq
and chose the sign of $c_{n} = \pm \l _{N_{n}}^{s_{1}}$ accordingly.
Then, the first two terms in eq. \ref{eq. norm dist ag cob} are, in absolute value,
$\geq | c_{n} \langle u_{n} ,\w _{n}  \rangle |$ so that, under our assumptions,
\begeq
| 
\d _{n}^{-1} \sum _{k<n} c_{k} \d _{k} \langle u_{n} , \w _{n} \rangle
+
c_{n} \langle u_{n} ,\w _{n}  \rangle | \geq 
| c_{n} \langle u_{n} ,\w _{n}  \rangle | \geq K \l _{N_{n}}^{s_{1}-s_{0}} \ra \infty
\endeq
so long as $s_{1}>s_{0}$.

In order to establish the divergence of the limit in eq.
\ref{eq. distr. ag cob}, we have to be able to conclude that the last sum in eq.
\ref{eq. norm dist ag cob} is $o(\l _{N_{n}}^{s_{1}-s_{0}})$.
To this end, it actually suffices to estimate brutally
$ | c_{k} \langle u_{n} ,\w _{k}\rangle | \leq | c_{k} |= \l _{N_{k}}^{s_{1}}$,
which implies that
\begeq
|\d _{n}^{-1} \sum _{k>n} c_{k} \d _{k} \langle u_{n} , \w_{k} \rangle | \leq
\d _{n}^{-1} \sum _{k>n} |c_{k}| \d _{k}
\endeq
Then, up to considering a subsequence, we can bound the rhs of the inequality
by an absolute constant, and this concludes the construction of $\w (\. )$.
\end{proof}

We remark that the proposition shows that, under quite mild conditions, $Cob^{\infty}(f)$ may not
even contain the space
\begeq
\{ \f \in \smm , \langle u_{n} , \f \rangle \ra 0 \}
\endeq
We also remark that we do not need to assume that $N_{n} \ra \infty$, even 
though this is expected to occur in general. In particular, if the limit diffeomorphism
$f$ is $DUE$, the sequence $N_{n}$ has to diverge, as imply the following lemma and its
corollary.

\begin{lemma} \label{lem approx dist comp}
Let $f \in \Diffm $, and suppose that there exists a sequence $\{ f_{n} \} \in \Diffm $,
$f_{n} \ra f$ and satisfying the following condition:
\begin{enumerate}
\item $f_{n} \in \dot{\HH} ^{-\s } (M) $
\item there exists a pre-compact sequence $(u_{n}) \in \dot{H}^{-\s} (M)$ with
$u_{n } \in \dot{H} ^{-\s } (f_{n})$ for every $n \in \N$.
\end{enumerate}
Then, $f \notin DUE(M)$.
\end{lemma}

\begin{proof}
%Let $u \in \overline{\{u_{n} \}}^{cl_{H^{-\s}}} \subset \dot{H}^{-\s } (M) $, and let $\f \in \smm (M)$ be such that
Let
\begeq
u \in \bigcap _{n \in \N} \overline{\bigcup _{k \geq n} \dot{H}^{-\s } (f_{k})}^{cl_{H^{-\s}}} \subset \dot{H}^{-\s } (M)
\endeq
be an accumulation point of the invariant distributions of the $f_{n}$, and let $\f \in \smm (M)$ be such that
$\langle u , \f \rangle = 1$. If $\f$ were an approximate coboundary over $f$, then there would
exist $\epsilon (\. )$ and $\psi (\. )$ such that
\begeq
\psi \circ f (\. ) - \psi (\. ) = \f (\. ) +\epsilon (\. )
\endeq
with $\epsilon \in \smm $ arbitrarily small in the $\sm $ topology, and $\psi \in \smm$ depending on
$\epsilon (\. )$. Given such $\epsilon (\. )$ and $\psi (\. )$, for $n $ big enough
\begeq
\psi \circ f _{n} (\. ) - \psi (\. ) = \f (\. ) + \epsilon (\. ) + \d _{n} (\. )
\endeq
with $ \d _{n} (\. )= \psi \circ f _{n} (\. ) - \psi \circ f (\. )$ arbitrarily small in the $\sm $ topology
as $n \ra \infty$. For $n$ big enough and in a
subsequence, the rhs tests $>1/2$ against $u_{n}$, while the lhs test $0$, a contradiction.
\end{proof}
In fact, $u \in \dot{H}^{-\s}(f)$. The proof grants the following corollary.
\begin{corollary} \label{cor div Nn}
Under the hypotheses of prop. \ref{prop approx}, if additionally $f \in DUE(M)$, then $N_{n} \ra \infty$.
\end{corollary}
\begin{proof}
If $N_{n} $ can be chosen to be bounded by $M \in \N ^{*}$, then there exists $\f \in \smm (M)$,
spectrally supported in the first $M$ modes,
\begeq
T_{M} \f (\. ) = \f (\. )
\endeq
and such that
\begeq
\limsup | \langle u _{n}, \f \rangle | > 0
\endeq
The proof of lemma \ref{lem approx dist comp} implies that $\f $ cannot be an approximate coboundary
over $f$.
\end{proof}
\begin{remark}
We point out that we have shown that, if there exists $\f \in \smm (M)$, a sequence
$f_{n} \ra f$ and $u _{n} \in \dot{H}^{\s}(f _{n})$, $\| u \|_{-\s} =1$, such that
\begeq
\limsup | \langle u _{n}, \f \rangle | > 0
\endeq
then $f \notin DUE$.
\end{remark}

Lemma \ref{lem approx dist comp} and its corollary imply the intuitively obvious fact that if
a diffeomorphism $f$, built by approximation by the sequence $f _{n}$, is to be $DUE$, then
the obstructions of the $f_{n}$ have to recede to infinity, precisely as in \cite{AFKo2015} and
\cite{NKInvDist}.
The important point in the proof is of course the compactness of functions with bounded spectrum.

\subsection{Proof of thm \ref{thm inst AK}}

We now provide the proof of the precise version of thm \ref{thm inst AK}.

%Let us now proceed to the proof and begin by proving the following lemma.
%
%\begin{lemma}
%Suppose that $\SSS \subset \Diffm $ is closed and that $\dot{H}^{-\s}(M)$ is dense $\SSS$.
%Then, $\SSS \setminus CR(M)$ is of the second Baire category.
%\end{lemma}

\begin{proof} [Proof of thm \ref{thm inst AK precise}]
Let $\{ f_{n} \}_{\N}$ be a dense set in $\dot{\HH }^{-s}(M) \cap \SSS$. Let, now, $i,k \in \N ^{*}$,
and for $n \in \N$ choose $u_{n} \in \dot{H}^{-\s}(f_{n})$ of norm $1$ and
$N_{n} = N_{n}(i,k)$ such that
\begeq
\|T_{N_{n}} u_{n} \| _{H^{-\s}} \geq k ^{-1} \l _{N_{n}}^{-i}
\endeq
For each $n$, choose $u_{n}$ so that $N_{n}$ is minimal.
Then, $N_{n}$ is non-increasing as $i,k$ increase.

Now, for $j,l \in \N ^{*}$ define
\begeq
V_{j,l} = V_{j,l}(i,k) =  \{ f \in \SSS , \exists n \in \N , d_{\s} (f, f_{n}) < l ^{-1} \l^{-j}_{N_{n}} 
\}
\endeq
By definition, $V_{j,l}$ is open and non-empty, since it contains $\{ f_{n} \}$.

Lemma \ref{badness estimate} shows that $\SSS \setminus CR(M)$ contains the $G_{\d}$ set
\begeq
\bigcup _{i,k} \bigcap _{j,l} V_{j,l}(i,k)
\endeq
and thus $\SSS \setminus CR(M)$ is of the second category.
\end{proof}
%
%The proof of thm \ref{thm inst AK precise} is now straightforward.
%
%\begin{proof} [Proof of thm \ref{thm inst AK precise}]
%Under the assumptions of the theorem, $\dot{H}^{-\s}(M)$ is dense in $\SSS $.
%By the previous lemma, $\SSS \setminus CR(M)$ is non-empty and of second category.
%Consequently, $CR(M) \cap \SSS $ is a meagre set in $\SSS$.
%\end{proof}

Lemma \ref{badness estimate 0} becomes relevant in a space where Cohomological Instability
%$DUE(M)\setminus CR(M)$
is
a priori known to be dense. The following theorem shows that it is sufficient for $CR$
to be meagre, independently of the presence of distribution-preserving diffeomorphisms.

Before stating the theorem, we need to establish some notation.
\begin{definitionn}
Let us call for
$-\infty \leq \s \leq \infty$
\begeq
\QQ (\s ) = \{
f \in \Diffm, %\overline{Cob^{\s}(f)}^{cl_{\infty}} = \smm (M) \text{ and }
Cob^{\s}(f) \subsetneq \overline{Cob^{\s}(f)}^{cl_{\infty}} \subseteq \smm (M)
\}
\endeq
\end{definitionn}
We have clearly $\QQ (\s ') \subset \QQ (\s )$ if $\s < \s '$. Moreover, as soon as
the the coboundary operator is not open, its image is meagre in $\smm (M)$ (this is
true for linear operators in Banach spaces).
A Diophantine rotation is not in any $\QQ (\s )$, while a Liouvillean one is in
$\QQ (\s ) $ for all $-\infty \leq \s \leq \infty$. A non-irrational rotation in $\T ^{d}$
inducing a minimal rotation in a torus of dimension $0<d' < d$ will be in all, resp. no,
$\QQ (\s )$, if the induced rotation is Liouville, resp. Diophantine.
%
%\begeq
%\QQ (\s ) = \{
%f \in DUE(M), %\overline{Cob^{\s}(f)}^{cl_{\infty}} = \smm (M) \text{ and }
%Cob^{\s}(f) \subsetneq \smm (M)
%\}
%\endeq

\begin{theorem} \label{thm pert DUE}
Let $\SSS$ be a closed subset of $\Diffm $ and suppose that there exists $\s >- \infty$
such that $\QQ (\s ) \cap \SSS $ is dense in $\SSS$.
Then $CR(M)\cap \SSS$ is meager in $\SSS $ (or empty).
\end{theorem}

Naturally, the theorem becomes relevant when $\QQ (\s ) \cap DUE $ is a priori known to
be dense in $\SSS$, for otherwise this theorem is a weaker version of thm.
\ref{thm inst AK precise}. This is a good moment to remark that we do not know whether
\begeq
\QQ (\s ) \subset \overline{\HH ^{-\s }(M)}^{cl_{\infty}}
\endeq
is true, i.e. whether Cohomological Instability is caused by approximation of
distribution-preserving diffeomorphisms. We now come to the proof of thm.
\ref{thm pert DUE}

\begin{proof}
%Let $\{f_{n} \}_{n \in \N} \subset DUE(M)\setminus CR(M)$ be dense in $\SSS$.
Let $\{f_{n} \}_{n \in \N} \subset \QQ (\s )$ be dense in $\SSS$. Then, there exists
a dense subset $\{ \f _{n} \}_{\N} \subset \smm (M)$, such that
\begeq
\f _{n} \notin Cob ^{\s} (f_{m} ) , \forall n,m \in \N
\endeq
This follows from the fact that the image of a linear operator (in our case of the coboundary operator
$\psi \mapsto \psi \circ f - \psi $, for any fixed $f \in \Diffm $) is either the full space
or meagre. By the hypothesis, we can chose $\s$ uniformly in $n$ so that $Cob^{\s } (f_{n} ) $
be meagre in $\smm (M)$ for all $n \in \N$.

Lemma \ref{badness estimate 0} then implies that the set
\begeq
\VV _{M,n} =
\{
f \in \SSS , \psi \circ f - \psi = \f _{n} \Rightarrow \| \psi \|_{C^{\s+1}} > M
\}
\endeq
is open for $M,n \in \N$. We formally define $\| \psi \|_{C^{\s+1}} = \infty $ if no
such solution exists. Clearly, if we call $\RR $ the complement of $CR(M)$ in
$\SSS$,
\begeq
\bigcap _{M,n} \VV _{M,n} \subset \RR 
%\left( DUE(M) \setminus CR(M) \right) \cap \SSS
\endeq
Thus, the set $\RR  $ contains a $G_{\d}$ set and is assumed dense. It is
consequently generic.
\end{proof}

The condition that $\QQ	 (\s ) $ be dense in $\SSS $ for some $\s > -\infty$ is satisfied
by all known $DUE\setminus CR$ examples. It does not seem to be implied by mere density of
$DUE$, and this is indicated by the following lemma, or rather by its proof. We remark
that the hypothesis demands that $f $ not be $CS$, which is weaker than it not being
$CR$.

\begin{lemma}
Let $f \in \Diffm  \setminus CS(M)$. Then, there exists $ -\infty \leq s_{0}< \infty$ such that
$f \in \QQ (s)$ for all
$-\infty \leq s < - s_{0}$.
\end{lemma}

\begin{proof}
If $f \in \Diffm  \setminus CS(M)$, then there exists a sequence $n_{k} \nearrow + \infty$ such that, if we call
%\begeq
%\TT =\{ \sum _{k} \psi _{n_{k}} \left( \f _{n_{k}} \circ f (\. ) - \f _{n_{k}} (\. ) \right) ,
% \psi _{n_{k}} = O(\l _{n_{k}}^{-\infty })
%\} %^{cl_{\infty}}
%\endeq
\begin{eqnarray*}
\UU &=& \{ \sum _{k}\hat{\psi} _{n_{k}}  \phi _{n_{k}} (\. )  ,
 \hat{\psi} _{n_{k}} = O(\l _{n_{k}}^{-\infty })
\} \subset \smm (M) \\
\VV &=& \{ \sum _{n \notin \{ n_{k}\}} \hat{\psi} _{n}  \phi _{n} (\. )  ,
 \hat{\psi} _{n} = O(\l _{n}^{-\infty })
\} \subset \smm (M)
\end{eqnarray*}
then
\begin{enumerate}
\item 
the space of coboundaries over functions in $\UU$ is not closed:
\begeq
\UU (f) = \{ \psi \circ f - \psi , \psi \in \UU \} \subsetneq \overline{\UU (f)} ^{cl_{\infty}} \hra \smm (M)
\endeq
\item with the same notational convention,
%\begeq
%\QQ (f) = \{ \psi \circ f - \psi , \psi \in \QQ \} \subsetneq \overline{\TT (f)} ^{cl_{\infty}} \hra \smm (M)
%\endeq
 $\mathrm{codim} \left( \overline{\VV (f)}^{cl_{\infty}} \right) = \infty$ in $ \overline{\{ Cob^{\infty} (f) \}}
 ^{cl_{\infty}}$
\item There exists $s_{1} < \infty $ such that for every $s \geq 0$,
\begeq
\| \phi _{n_{k}} \circ f (\. ) - \phi _{n_{k}} (\. ) \| _{s} = O ( \l _{n_{k}}^{ s_{1}})
\endeq
\end{enumerate}
The third item is due to the fact that, if the polynomial rate of growth of each
$C^{s}$ norm is not bounded uniformly\footnote{Uniformity is in $s$, not in the constant
involved.} for $s$, i.e. if there does not exist such an $s_{1}$, then
\begeq
\sum _{k} \hat{\psi} _{n_{k}} \left( \phi _{n_{k}} \circ f (\. ) - \phi _{n_{k}} (\. ) \right)
\endeq
converges to a $\sm$ function iff $\hat{\psi} _{n_{k}} = O(\l _{n_{k}}^{-\infty })$,
i.e iff the function
\begeq
\sum _{k}\hat{\psi} _{n_{k}} \phi _{n_{k}} (\. )
\endeq
is itself $\sm$ which implies that $\UU (f)$ is closed.

The conclusion of the lemma now follows as in the case of Liouvillean rotations.
The space $Cob^{s} (f) $ is meager in $Cob^{s_{1}} (f) $ whenever $s<s_{1}$, since
\begeq
\{ \hat{\psi} _{n_{k}} = O(\l _{n_{k}}^{s})
\} 
\endeq
is meagre in
\begeq
\{ \hat{\psi} _{n_{k}} = O(\l _{n_{k}}^{s_{1}})
\} 
\endeq
\end{proof}

The reason why the condition of thm. \ref{thm pert DUE} does not seem to be implied by
density of $DUE$ in $\SSS$ is the following. Cohomological instability is caused by
an at most polynomial growth of the $C^{s}$ norms of the "elementary coboundaries"
\begeq
\phi ^{f}_{n_{k}} (\. ) = \phi _{n_{k}} \circ f (\. ) - \phi _{n_{k}} (\. )
\endeq
along a subsequence $\{ n_{k} \}$ depending on $f$. A generic function has full spectrum,
and, given two $DUE \setminus CR$ diffeomorphisms $f$ and $f'$, arbitrarily close in
$\sm$, a generic function in $\smm$ will not be an exact coboundary over either of the two.
However, the norms of the elementary coboundaries
\begeq
\phi ^{f'}_{n_{k}} (\. ) = \phi _{n_{k}} \circ f' (\. ) - \phi _{n_{k}} (\. )
\endeq
over $f'$ might grow fast, so that instability of $f'$ may be caused by the slow growth
along a different subsequence $\{ n'_{k} \}$. In that case, trying to control
$\f ^{f'}_{n'_{k}} (\. ) $ using information on $\f ^{f}_{n_{k}} (\. )$ seems hopeless in general.

In the known $DUE \setminus CR$ examples of Liouvillean rotations and of $DUE$ cocycles
in $\T ^{d} \times SU(2)$, the rate of growth (actually decay) is constant and equal to
$\s = -\infty$, but instability can be caused by the fast decay of the norms of
elementary coboundaries along different subsequences.

Given the above, the following question seems interesting and beyond
the scope of present technology.
\begin{question}
Does there exist a $DUE$ diffeomorphism $f \in \Diffm $ of a compact manifold $M$
for which there exists $\s \in [-\infty , \infty ) $ such that
% $Cob ^{\s } (f)$ is meagre in $\smm (M)$,
%and a sequence $\{ f_{n} \}$, $f_{n} \neq f$, $f_{n} \ra f $ in the $\sm $ topology,
%$f_{n} \in DUE(M) \setminus CR(M)$ and such that for every $n$
\begeq
Cob^{\s } (f) = \smm (M)
\endeq
\end{question}
%
%In the existing $DUE$ examples, the approximating diffeomorphisms $f_{n}$ do not only
%preserve distributions, but also functions.

\subsection{Some comments on the proof}

Proposition \ref{prop approx} shows
that fast approximation of a diffeomorphism by diffeomorphisms that preserve distributions
creates an obstruction to certain functions being exact coboundaries. An Anosov-Katok
type argument can make sure that the obstructions of $f_{n}$ recede to infinity as $n \ra \infty$
and $f_{n} \ra f$, which would make every function in $\smm (M)$ an approximate coboundary.
On the other hand, theorem \ref{thm inst AK precise} shows that the limit procedure leaves
a trace, and, for a generic diffeomorphism $f$ obtained in this way, a generic function will not
be an exact coboundary.

Since we are working in the measure preserving category, we also have a canonical way of
identifying functions with distributions, via the $L^{2}$ self duality. If, now, one wishes to
obtain a weak solution of eq. \ref{lin cohom eq} in some space of distributions $\dot{H}^{-s}$,
then the existence of a non-zero function $\w \in H^{s}_{\mu }$, invariant under $f$, poses
an obstruction, as the rhs function $\f$ has to satisfy
\begeq
\langle \w , \f \rangle _{L^{2}} = 0
\endeq
for it to be in the range of the coboundary operator $\psi \mapsto \psi \circ f - \psi$
on $\dot{H}^{-\s} \ra \dot{H}^{-\s}$.
For the same reason, fast approximation of $f$ by diffeomorphisms $f_{n}$
preserving functions in any $\dot{H}^{s}$, would imply that functions $\f \in \smm $ that are
coboundaries in the sense of distributions,
\begin{eqnarray*}
Cob^{-\s} (f) &=& \{
\f \in \smm , \exists u \in \dot{H}^{-\s} , f_{*}u - u = \f \text{ in } H^{-\s}
\} \\
&=& \{
\f \in \smm , \exists u \in \dot{H}^{-\s} , \langle u , \psi \circ f - \psi \rangle =
\langle \f , \psi \rangle , \forall \psi \in \smm
\}
\end{eqnarray*}
will still be meagre in $\smm $, for a generic such $f$ and for every $\s \in [0,\infty )$.

\section{Proof of corollary \ref{cor no counter-ex}} \label{section proof cor}

In this section we provide justification for our claim in corollary \ref{cor no counter-ex}
that the Anosov-Katok construction is not fit for providing counter-examples to the conjecture. In what follows, we use the notation established in \S \ref{subsec AK}.

Anosov-Katok constructions are based on arithmetic properties. These properties are
usually absolute, as in the original construction \cite{AnKat1970}. Alternatively, when
the construction is fibered over a Diophantine rotation, as in \cite{NKInvDist}, these
arithmetic properties need to be relative to the fixed rotation over which the
construction fibers.

From the spectral point of view, what both cases have in common can be described as follows.
We work in the closure of systems that are conjugate to constant ones.
In this space, differentiable rigidity of some examples is known, namely of the examples satisfying some
Diophantine property. The coboundary space of such systems is small, due to the presence
of rigidity. In order to construct systems whose coboundary space is maximal, one starts
with systems whose coboundary space is as small as possible, and then by an
approximation-by-conjugation scheme tries to make the coboundary space maximal when the
limit is reached. The role of systems with small coboundary space is played by periodic rotations in
the original setting, \cite{AnKat1970}, and by resonant cocycles (see definition
\ref{def res coc}) in the fibered setting, \cite{NKInvDist}. The approximation-by-conjugation
scheme destroys periodicity and reducibility to a resonant constant respectively by making
the conjugation degenerate and the period or the resonant coefficient go to infinity.
For the scheme to converge, some fast approximation conditions are needed, which
result in the Liouville character of the dynamics obtained in the limit.

In the setting of the present article, the manifold $M$ is not a priori a Lie group, and
thus no notion of arithmetics is pertinent. Since, however, theorem \ref{thm inst AK precise}
provides a way of determining whether a system is $CR$ by looking at its distance from
systems preserving distributions, we can make an abstraction of the Anosov-Katok scheme
as presented above.

\subsection{The Anosov-Katok-like construction}

The goal of an Anosov-Katok-like construction is to start with a space of diffeomorphisms
$\CC$, whose cohomological propserties are well understood, and construct $CR$ examples
in $\AKsm (\CC )$. We will impose conditions that favour genericity of $DUE$ in
$\AKsm (\CC )$, and show that under such conditions, proving the existence of $CR$
diffeomoprhisms via an approximation-by-conjugation scheme with the caracteristics of
the Anosov-Katok method will be impossible.

Firstly, the fact that the unit sphere of a finite dimensional space is compact, together with
lemma \ref{lem approx dist comp} imply that if diffeomoprhisms in $\CC$ preserve at most
a finite-dimensional space of distributions, the conclusion of the construction
$\AKsm (\CC )$ is trivial. For this reason, we introduce the first standing assumption
on the construction.
\begin{assumption} \label{ass inf dim ob}
We suppose that mappings in $\CC$ are Cohomologically Stable, and that each such mapping
preserves infinitely many linearly independent distributions of
bounded regularity $\s $:
\begeq
\exists \s \geq 0 \text{ such that }
h \in \CC \Rightarrow \dim \dot{H} ^{-\s } (h) = \infty
\endeq
Clearly, the same holds for $\TT$.
\end{assumption}
In order to avoid trivialities, we also impose the following assumption, without the
uniformity of the previous one.
\begin{assumption}
We suppose that mappings in $\overline{\CC }$ are not DUE.
\end{assumption}
As a consequence, if the sequence of conjugations constructed in the scheme converges,
then the mapping obtained in the limit will not be $DUE$. Later on,
we will make an assumption on the inverse of the coboundary operator for diffeomorphisms
in $\CC$, which we will not impose in $\overline{\CC}$. In particular these assumptions
cover the case where $\CC$ is left translations by rational points in some Lie group,
which thus cannot be a torus.

If $\CC $ is a class of diffeomorphisms serving as a basis for an Anosov-Katok space
$\AKsm (\CC )$, an Anosov-Katok construction does not access every element of $\AKsm (\CC )$, that is, there exist $f \in \AKsm (\CC )$ that are not limits of the considered
Anosov-Katok construction. In the context of the original construction, such $f$ would be
Diophantine rotations around the centre of the disk $\mathbb{D}$. This is so because
typically a condition of fast convergence is necessary for the construction to
converge. Let us formalize this by imposing the following assumption.

We recall the notation
\begeq
f_{n} = H_{n} \circ \tilde{f}_{n} \circ H_{n}^{-1}
\endeq
with $\tilde{f}_{n} \in \CC$.

The notion of \textit{fast convergence} of $f_{n} \in \TT $ to $f \in \AKsm (\CC )$ is
specified in the following way, where we introduce the notation
\begeq
\eta _{s ,n}((g_{n})) =d_{s} (g_{n-1}, g_{n})
\endeq
for a sequence $g_{n} \ra g$. In the context of the Anosov-Katok construction, we will
use the notations
\begeq
\eta _{s ,n}(f)= \eta _{s ,n} =d_{s} (f_{n-1}, f_{n}) \text{ and } \tilde{\eta} _{s ,n} =
d_{s} (\tilde{f}_{n-1}, \tilde{f}_{n})
\endeq
for the admissible size of the perturbation at the $n$-th step of the construction,
whenever $f_{n}\ra f$.
\begin{definitionn} \label{def fast approx}
We will say that $g_{n} \ra g $ fast if for every $s$, $\eta _{s ,n}((g_{n})) =
O(\eta _{s ,n-1}((g_{n}))^{\infty})$, i.e.
\begeq
\eta _{s ,n-1}^{-l}((g_{n})) \eta _{s ,n}((g_{n})) \ra 0 , \forall l \in \N
\endeq
\end{definitionn}

\begin{assumption} \label{ass fast conv}
If $\CC \subset \dot{\HH}^{-\s}(M)$, then an Anosov-Katok-like construction upon $\CC $
satisfies the fast convergence assumption for $\tilde{f}_{n}$,
i.e. whenever $f_{n}$ is a convergent sequence constructed by the scheme, the
corresponding representatives in $\CC$ converge fast to their limit for all $s$:
\begeq
\tilde{\eta} _{s,n-1}^{-l} \tilde{\eta} _{s,n} \ra 0 , \forall l \in \N
\endeq
\end{assumption}

The fast approximation assumption for the representatives in $\CC$ comes along with some
restriction on the size of the conjugations used in the construction, so that $f_{n}$
converges in the $\sm$ topology. Let us call
\begeq
\e _{s,n} = \| h_{n}\|_{s} \text{ where } h_{n} = H_{n}^{-1}\circ H_{n+1}
\endeq
the size of the increment of the conjugation at the $n$-th step.
\begin{assumption} \label{ass size conj}
At the $n$-th step of the construction the admissible size of the conjugation is
polynomial with respect to $\tilde{\eta}$: there exists $\nu \geq 0 $ such that
for all $s $ and for all $n$
\begeq
\e _{s +1,n} \leq K_{s} \tilde{\eta}_{s ,n-1}^{-\nu}
\endeq
\end{assumption}
In view of eq. \ref{eq increment AK}, this is a sufficient condition for $f_{n}$ to
converge in the $C^{\infty} $ topology.

Let us recall the notation
\begeq
\d _{s ,n}=d_{s} (f_{n}, f)
\endeq
and state the following easy lemma.
\begin{lemma} \label{lem orders magn}
Under assumptions \ref{ass fast conv} and \ref{ass size conj},
\begin{enumerate}
\item $\eta _{s ,n} = O(\tilde{\eta }_{s,n-1}^{-\infty})$
\item $\eta _{s ,n} = O(\tilde{\eta }_{s,n}^{1-\e})$, for every $\e >0$
\item $\d _{s,n} = O(\tilde{\eta} _{s ,n-1}^{\infty})$
\item $\d _{s ,n} = O(\tilde{\eta }_{s,n}^{1-\e})$, for every $\e >0$
\end{enumerate}
\end{lemma}
\begin{proof}
\begin{enumerate}
\item By definition,
\begeq
\eta _{s ,n} = d_{s}(H_{n}\circ \tilde{f}_{n}\circ \tilde{f}_{n-1}\circ \tilde{H}_{n}^{-1}, \Id)
\endeq
By the definition of $H_{n}$ and the assumptions of the lemma,
\begeq
d_{s}(H_{n}, \Id) = O(\tilde{\eta}_{s ,n-1}^{-\nu})
\endeq
which implies the first item.
\item The second item follows from the proof of the first one:
\begeq
\begin{array}{r@{}l}
\eta _{s ,n} &= O(\d _{s,n}\tilde{\eta}_{s,n})\\
&= O(\d _{s,n}\tilde{\eta}_{s,n}^{\e} \tilde{\eta}_{s,n}^{1-\e})\\
&= O(\tilde{\eta}_{s,n}^{1-\e})
\end{array}
\endeq
\item This follows from the triangle inequality,
\begeq
\d _{s,n} = O(\sum _{n}^{\infty}\eta _{s,n})
\endeq
and the fact that $\eta_{s,n+1} = O(\eta_{s,n}^{\infty})$.
\end{enumerate}
\end{proof}

This ends the preliminary assumptions that define what we call an Anosov-Katok-like
construction. They concern only rates of convergence, and not cohomology. The
cohomological properties of diffeomorphisms in $\CC$ will be adressed in the following
paragraph.

\subsection{Relation with proposition \ref{prop approx}}

Prop. \ref{prop approx} implies that, for $f$ to be $CR$, approximation should not be
fast with respect to some subsequence of eigenvalues $\l _{N_{n}}$ related to $f_{n}$.
Our assumption in this paragraph is that we can treat the sequence $f_{n}$ as
diffeomorphisms analyzable only through their definition as
\begeq
f_{n} = H_{n}\circ \tilde{f}_{n} \circ H_{n}^{-1}
\endeq
In other words, only the cohomological properties of $\tilde{f}_{n} \in \CC$ and the size
of the $H_{n}$ are known, while the properties of the $f_{n}$ can only be deduced from
the conjugation relation. This reflects the constructivity of the approach and puts us
slightly outside the scope of proposition \ref{prop approx}: we are not able to compare
$\d _{\s } = d_{\s} (f_{n},f)$ with the spectrum of the distributions preserved by
$f_{n}$.

We thus fix constants $\s$, $s_{0}$ and $K$. As a consequence, to each $\tilde{f} \in \CC$
we associate a truncation order $N(\tilde{f}) \in \N ^{*}$ such that eq. \ref{eq spec un}
is verified. The truncation orders along a sequence $\tilde{f}_{n}$ will be denoted by
$N_{n}$ instead of $N(\tilde{f}_{n})$. Anosov-Katok-like constructions depend on the choice
of these parameters (cf. remark \ref{rem depend on param}).

We introduce the following assumption, constraining convergence so that it not be slow
with respect to $\l _{N_{n}}$. The condition is on the $\tilde{f}_{n}$. It implies that
our estimates work with rates of convergence outside the scope of proposition
\ref{prop approx}, where $CR$ is possible, at least a priori.
\begin{assumption}  \label{ass decay eta}
We suppose that, for the construction to converge, the sequence $f_{n}$ has to satisfy
the following.
For any choice of the sequence $\{ N_{n} \}$ as in prop. \ref{prop approx} applied to
$\tilde{f}_{n}$, there exists $\t >0 $ and
%such that for every $s \in \N$ as in def. \ref{def fast approx}, there exists
$\gamma = \gamma _{\s} >0$ such that
\begin{equation} \label{eq slow decay d}
\tilde{\eta} _{s ,n} \leq \gamma \l _{N_{n }}^{-\t }
\end{equation}
\end{assumption}

\begin{remark}
In view of remark \ref{rem depend on param}, as $s_{0}$ increases and $K$ decreases, for
$\gamma $ and $\t$ kept fixed, more perturbations are authorized under assumption
\ref{ass decay eta}.
\end{remark}

As a consequence, proposition \ref{prop approx} cannot be applied in order to show
that $\tilde{f} = \lim \tilde{f}_{n} \in \overline{\CC}$ is not Cohomologically Rigid.
We actually want to authorize $CS \cap \bar{\CC} \neq \emptyset$ in order to favour
$CR \cap \AKsm (\CC ) \neq \emptyset$.

Our standing assumptions allow us to exclude $CR$ in a first kind of constructions. We call
this case "compact group case" for reasons that we will explain in \S \ref{sec AR}.

\subparagraph*{The compact group case} This is the case where, for infinitely many steps,
the construction fails to destroy for all obstructions arising at the step $n$ in the
$n+1$-th step, but needs to wait for arbitrarily many steps of the algorithm. In this case,
the orders of magnitude allow us to conclude that the limit object $f$ cannot be $CR$.

The relevant assumption is the following, which allows us to show that carrying an
obstruction even for one step precludes Cohomological Rigidity, as long as the fast
convergence assumption is verified.
\begin{assumption} \label{ass comp gr case}
An Anosov-Katok-like construction satisfies the assumptions of the compact group case iff
for any admissible sequences $\tilde{f}_{n}$ and $h_{n}$, the following is true, at least
along a subsequence $n_{k} \ra \infty$. For any given $\tilde{f}_{n}$ and any admissible
choice of $\tilde{f}_{n+1}$ and $h_{n}$, there exists
\begeq
\tilde{u} _{n} \in \dot{H}^{-\s }(\tilde{f}_{n})
\text{ such that }
\| T_{N_{n}} \tilde{u}_{n}  \|_{H^{-\s }} \geq K \l _{N_{n}} ^{-s_{0} }
\endeq
such that
\begeq
(h_{n})_{*}\tilde{u}_{n} \in \dot{H}^{-\s }(\tilde{f}_{n+1})
\endeq
The sequences $\tilde{f}_{n}$ and $h_{n}$ are admissible iff they satisfy assumptions
\ref{ass inf dim ob} to \ref{ass decay eta}.
\end{assumption}
We can now prove the following proposition.
\begin{proposition} \label{example prop comp gr}
Let an Anosov-Katok-like construction satisfy all standing assumptions \ref{ass inf dim ob}
to \ref{ass comp gr case}, so that it corresponds to the compact group case.

Then, the limit objects $f$ obtained by the construction are not $CR$.
\end{proposition}

The proof goes exactly as that of prop. \ref{prop approx}.

\begin{proof}
Let $\tilde{u}_{n}$ satisfy assumption \ref{ass comp gr case},
and define $\tilde{\w} _{n} \in \smm$ as in the proof of prop. \ref{prop approx}.

Suppose that $f \in CR$, and call $\psi_{n} \in \smm$ the solution to the cohomological
equation for $\tilde{\w} _{n} \circ H_{n+1} $ over $f$. If one tests the equation against
$u_{n+1} = (H_{n+1})_{*}\tilde{u}_{n}$ and uses invariance under $f_{n+1}$, then, just as
in the proof of lemma \ref{badness estimate}, one finds that
\begeq
\langle u _{n+1},
\psi_{n} \circ ( f \circ f _{n+1} ^{-1}) -  \psi_{n} \rangle =
 \langle \tilde{u}  _{n},\tilde{\w} _{n}  \rangle
\endeq
which implies that
\begeq
\|u_{n+1}\|_{H^{-\s}} \| \psi_{n} \|_{\s+1} \geq K \l _{N_{n}}^{-s_{0}}\d_{n+1,\s}^{-1}
\endeq
The estimate of eq. \ref{composition} and the fact that we can assume without any loss
of generality that $\|\tilde{u}_{n}\|_{H^{-\s}} = 1$, imply that
\begeq
\|u_{n+1}\|_{H^{-\s}} \leq C_{\s} \vertiii{H_{n+1}}_{\s}
\endeq
Therefore, we immediately have
\begin{equation} \label{eq est sol comp gr case}
\| \psi_{n} \|_{\s+1} \geq K_{\s} \l _{N_{n}}^{-s_{0}}\vertiii{H_{n+1}}_{\s}^{-1}\d_{n+1,\s}^{-1}
\end{equation}
Now, lemma \ref{lem orders magn} allows us to conclude: we have
\begeq
\vertiii{H_{n+1}}_{\s}^{-1} \gtrsim \tilde{\eta}_{\s ,n}^{\nu},
\d _{\s ,n+1}^{-1} \gtrsim \tilde{\eta}_{\s ,n+1}^{-1+\e},
\l _{N_{n}}^{-s_{0}} \gtrsim \tilde{\eta}_{\s ,n}^{s_{0/\t}}
\endeq
and $\tilde{\eta}_{\s,n}\ra 0 $ fast, so that $\|\psi _{n} \|_{\s+1} \ra \infty$ while
$\tilde{\w}_{n}$ stays bounded.
\end{proof}

We remark that the relevant part of the assumptions is that $\tilde{u}_{n}$ has significant support
in the modes smaller than $\l _{N_{n}}$. The fact that is it preserved by $\tilde{f}_{n}$
is actually irrelevant for the proposition. The fact that $h_{n}$ commutes with
$\tilde{f}_{n}$, and thus is an one-to-one automorphism of $\dot{H}^{-\s }(\tilde{f}_{n})$
is not relevant, either.

They are both relevant, however, in the context of the construction: when the
construction tunes the dynamics in the scale $\l_{N_{n+1}}^{-1}$, which is hopefully a lot
smaller than $\l_{N_{n}}^{-1}$, one does not want to create obstructions at the scale
$\l_{N_{n}}^{-1}$ that did not exist before. The proposition implies that one should not
carry over any obstructions from the previous scale, either.

\subparagraph*{The rotation vector case}
We are now left with the case where, even though
\begin{equation} \label{eq inters obs}
\dot{H}^{-\s }(\tilde{f}_{n}) \bigcap \dot{H}^{-\s }(\tilde{f}_{n+1}) \neq \emptyset
\end{equation}there exist admissible conjugations $h_{n}$ violating assumption \ref{ass comp gr case}.
We note that the set of such conjugations would be open if $\dot{H}^{-\s }(\tilde{f}_{n})$
were of finite dimension, and is $G_{\d}$ under the assumption that each diffeomorphism in
$\CC$ preserve an infinite number of distributions (assumption \ref{ass inf dim ob}).

We remind that in the absence of the condition in eq. \ref{eq inters obs}, one need not
introduce a conjugation in order to destroy invariant distributions. Therefore, the
condition is natural in the context of an Anosov-Katok-like construction.

\begin{assumption} \label{ass rot vec case}
An Anosov-Katok-like construction satisfies the assumptions of the rotation vector case iff
for any admissible sequences $\tilde{f}_{n}$ and $h_{n}$, the following is true, at least
along a subsequence $n_{k} \ra \infty$. For any given $\tilde{f}_{n}$ and any admissible
choice of $\tilde{f}_{n+1}$, there exists an admissible conjugation $h_{n}$ such that for
all
\begeq
\tilde{u} _{n} \in \dot{H}^{-\s }(\tilde{f}_{n})
\text{ such that }
\| T_{N_{n}} \tilde{u}_{n}  \|_{H^{-\s }} \geq K \l _{N_{n}} ^{-s_{0} }
\endeq
we have
\begeq
\tilde{\w}_{n} \circ h_{n}\in Cob^{\infty} (\tilde{f}_{n+1})
\endeq
The sequences $\tilde{f}_{n}$ and $h_{n}$ are admissible iff they satisfy assumptions
\ref{ass inf dim ob} to \ref{ass decay eta}.
\end{assumption}

In this case where assumption \ref{ass comp gr case} is violated. We call it the
"rotation vector case" as it resembles the way invariant distributions behave when a
construction is based upon periodic translations in tori. In the rotation vector case,
will need more assumptions
than the ones on orders of magnitude, and more precisely on the inverse of the coboundary
operators of the diffeomorphisms in $\CC$. There is, however, a subcase where such
assumptions are not needed.

If convergence of $\tilde{f} _{n}\ra \tilde{f}$ is fast and moreover
\begin{equation} \label{eq exp dec eta}
\tilde{\eta} _{s ,n} = O(\l _{N_{n}}^{-\infty})
\end{equation}
in which case the argument proving proposition \ref{prop approx} applies to $\tilde{f}$,
one should not expect the action of conjugations to save cohomological rigidity. Indeed, we
can prove the following proposition.

\begin{proposition} \label{example prop rot vec 1}
Let an Anosov-Katok-like construction satisfy standing assumptions \ref{ass inf dim ob}
to \ref{ass decay eta} and assumption \ref{ass rot vec case}, so that it
corresponds to the rotation vector case.
Assume, moreover, that the rate of convergence imposed in eq. \ref{eq exp dec eta} is
satisfied.

Then, the limit objects $f$ thus obtained by the construction are not $CR$.
\end{proposition}

\begin{proof}
The argument is again by orders of magnitude as in proposition \ref{example prop comp gr}.

Let $\tilde{u}_{n}$ and $\tilde{\w} _{n} \in \smm$ as in the proof of prop.
\ref{prop approx}, suppose that $f \in CR$, and call $\psi_{n} \in \smm$ the solution to
the cohomological equation for $\tilde{\w} _{n} \circ H_{n} $ over $f$. If one tests the
equation against $u_{n} = (H_{n})_{*}\tilde{u}_{n}$ and uses invariance under $f_{n}$,
then, just as in the proof of proposition \ref{example prop comp gr}, one finds that
\begeq
\| \psi_{n} \|_{\s+1} \geq K_{\s} \l _{N_{n}}^{-s_{0}}\vertiii{H_{n}}_{\s}^{-1}
\d_{n,\s}^{-1}
\endeq
This is eq. \ref{eq est sol comp gr case}, with the only difference that we are using
invariance by $f_{n}$ instead of $f_{n+1}$.

As in the proof of proposition \ref{example prop comp gr}, $\d_{n,\s}^{-1/2}$ can be used
to balance the term $\vertiii{H_{n}}_{\s}^{-1}$ and the remaining $\d_{n,\s}^{-1/2}$ in
order to balance $\l _{N_{n}}^{-s_{0}}$ in view of eq. \ref{eq strictly pol dec}.
%Now, lemma \ref{lem orders magn} allows us to conclude: we have
%\begeq
%\vertiii{H_{n}}_{\s}^{-1} \gtrsim \tilde{\eta}_{\s ,n-1}^{\nu},
%\d _{\s ,n+1}^{-1} \gtrsim \tilde{\eta}_{s ,n}^{-1+\e},
%\l _{N_{n}}^{-s_{0}} \gtrsim \tilde{\eta}_{\s ,n}^{s_{0/\t}}
%\endeq
%and $\tilde{\eta}_{\s,n}\ra 0 $ fast, so that $\|\psi _{n} \|_{\s+1} \ra \infty$ while
%$\tilde{w}_{n}$ stays bounded.
\end{proof}

We now remark that, of the decay of $\tilde{\eta}  _{s ,n}$ is strictly polynomial, i.e.
\begin{equation} \label{eq strictly pol dec}
\tilde{\eta} _{s ,n} \geq \tilde{\gamma} \l _{N_{n }}^{-\t }
\end{equation}which is the inverse inequality of that in \ref{ass decay eta}, then we have fast growth
of $\l _{N_{n}}$, due to the assumed
fast convergence of $\tilde{f} _{k} \ra \tilde{f} $, implies directly the fast growth of
the $N_{n}$.

\begin{lemma} \label{lem fast growth Nn}
If $\tilde{f} _{n}\ra \tilde{f}$ fast and assumption \ref{ass decay eta} and condition of
eq. \ref{eq strictly pol dec} both hold, we have fast growth of the sequence $N_{n}$:
\begeq
\l _{N_{n}}^{-l}\l _{N_{n+1}} \ra \infty , \forall l \in \N
\endeq
\end{lemma}
\begin{proof}
Since $\tilde{\eta} _{s ,n}^{-l}\tilde{\eta} _{s ,n+1}^{-1} \ra \infty $. Under our
assumptions we have
\begeq
\tilde{\eta} _{s ,n}^{-l}\tilde{\eta} _{s ,n+1}^{-1} \lesssim
(\l _{N_{n}}^{-l} \l _{N_{n+1}})^{\t }
\endeq
which proves the lemma.
\end{proof}
Naturally, if the exponent in eq. \ref{eq strictly pol dec} is smaller than the one of
assumption \ref{ass decay eta}, the result remains true (only notation gets heavier).

In order to treat this case, which is the only remaining one in order to prove corollary
\ref{cor no counter-ex}, we need to introduce the following assumption on the coboundary
operators of diffeomorphisms in $\CC$. Put informally, it states that solving at the
$n+1$-th step for obstructions arising at the $n$th step, which we assume possible, costs
a positive power of $\l _{N_{n+1}}$, the scale of the construction at the $n+1$-th step.

\begin{assumption} \label{ass cob op C}
%For every $s$, there exist constants $K ,\k>0$ depending on $s$, for which the following
%holds.
There exist constants $K ,\k>0$ for which the following holds.

Let a sequence $\tilde{f}_{n}$ in the Anosov-Katok-like construction satisfying assumption
\ref{ass comp gr case}.

For every $n$, consider distributions $\tilde{u}_{n}$ satisfying eq. \ref{eq spec un}, and
functions $\tilde{\w}_{n}$ as in the proof of proposition \ref{prop approx}.

We suppose that for every $n$, there exists $\tilde{u}_{n}$ such that, for every admissible
conjugation $h_{n}$ as in assumption \ref{ass rot vec case}, any function $\psi _{n+1}$
solving
\begeq
\tilde{\psi} _{n+1} \circ \tilde{f}_{n+1} -\tilde{\psi} _{n+1} =
\tilde{\w}_{n} \circ h_{n}
\endeq
satisfies the estimate
\begeq
K ^{-1} \l _{N_{n+1}}^{\k} \|\tilde{\w}_{n} \circ h_{n}\|_{s} \leq 
\|\tilde{\psi} _{n+1}\|_{s} \leq
K \l _{N_{n+1}}^{\k}\|\tilde{\w}_{n} \circ h_{n}\|_{s}
\endeq
for $s = \s , \s +1$.
\end{assumption}
Again, we suppose that the exponents for the lower and upper bounds are equal only in
order to keep notation lighter.

\begin{remark}
We observe that $\k $, provided that it exists and is finite, is non-increasing when
the sequence $\l _{N_{n}}$ decreases, since decreasing $\l _{N_{n}}$ amounts to estimating
the restriction of the inverse of the coboundary operator of $\tilde{f}_{n}$ to a smaller
space.
\end{remark}

We can prove the following proposition, which concludes the ingredients for the proof of
corollary \ref{cor no counter-ex}.

\begin{proposition} \label{example prop rot vec 2}
Let an Anosov-Katok-like construction satisfy standing assumptions \ref{ass inf dim ob}
to \ref{ass decay eta} and assumptions \ref{ass rot vec case} and \ref{ass cob op C},
so that it corresponds to the rotation vector case.

Then, the limit objects $f$ obtained by constructions for which $\k < \t$ are not $CR$.
The parameters $\k $ and $\t$ are as in assumptions \ref{ass cob op C} and
\ref{ass decay eta}, respectively.
\end{proposition}

We will discuss the hypothesis that $\k <\t$ in the next section.

\begin{proof}
By proposition \ref{example prop rot vec 1}, we can suppose that the lower
bound of eq. \ref{eq strictly pol dec} holds, and lemma \ref{lem fast growth Nn} applies.

Let us denote by $\psi _{n+1}$ the solution to the equation
\begeq
\psi _{n+1} \circ f_{n+1} - \psi _{n+1 } = \tilde{\w}_{n}\circ H_{n}^{-1}
\endeq
This is equivalent to $\tilde{\psi}_{n+1 } = \psi _{n+1 } \circ H_{n+1}$ being a solution
to
\begeq
\tilde{\psi}_{n+1} \circ \tilde{f}_{n+1} - \tilde{\psi}_{n+1 } = \tilde{\w}_{n}\circ h_{n}
\endeq
By hypothesis, we have
\begeq
\|\tilde{\psi}_{n+1}\|_{\s} \approx \l _{N_{n+1}}^{\k} \|\tilde{\w}_{n}\circ h_{n}\|_{\s}
\endeq
which implies that
\begeq
\l _{N_{n+1}}^{\k} \|\tilde{\w}_{n}\circ h_{n}\|_{\s} \vertiii{H_{n}}_{\s}^{-1} \lesssim
\|\psi_{n+1}\|_{\s} \lesssim
\l _{N_{n+1}}^{\k} \|\tilde{\w}_{n}\circ h_{n}\|_{\s} \vertiii{H_{n}^{-1}}_{\s} 
\endeq
A reasoning by orders of magnitude shows that, while for any positive $\e$,
\begeq
\l _{N_{n+1}}^{-\e} \tilde{\w}_{n}\circ H_{n}^{-1} \ra 0
\endeq
in the $\sm $ topology,
\begeq
\l _{N_{n+1}}^{-\e} \psi _{n+1}
\endeq
diverges in the $C^{\s+1}$ topology. Since, however, $\k < \t$ implies that% for every $s$
\begeq
\l _{N_{n+1}}^{-\e} \| \partial \psi _{n+1} \| _{\s} \d _{\s ,n+1} \ra 0
\endeq
lemma \ref{lem closed graph} implies that $f$ cannot be $CR$.
\end{proof}

We now prove corollary \ref{cor no counter-ex}.
\begin{proof}[Proof of Corollary \ref{cor no counter-ex}]
Consider an Anosov-Katok-like construction satisfying assumptions \ref{ass inf dim ob}
to \ref{ass decay eta} and either assumption \ref{ass comp gr case} or assumptions 
\ref{ass rot vec case} and \ref{ass cob op C}.

Application of propositions \ref{example prop comp gr} or \ref{example prop rot vec 1} and
\ref{example prop rot vec 2} in the respective cases proves that the diffeomorphisms thus
constructed are not $CR$.

Since the assumptions are cover all possible cases of Anosov-Katok-like constructions,
actually tautologically so, the corollary is proved.
\end{proof}

\subsection{Some remarks on a constructive approach to cohomology}

Invoking lemma \ref{lem closed graph} in the proof of proposition
\ref{example prop rot vec 2}, which imposes the seemingly restrictive hypothesis that
$\k < \t$. This hypothesis is in reality not restrictive as soon as one takes a
constructive point of view in the attempted construction of $CR $ examples other than
Diophantine rotations in tori.

By the term constructive point of view, we mean the following, which is quite explicit in
\cite{NKInvDist}. We consider the class $\CC$ as a class of examples, i.e. a class of
diffeomorphisms whose dynamics, and in particular their coboundary operators, are well
understood and can be inversed in an explicit way. Diffeomorphisms in $\TT $ are conjugate
to diffeomorphisms in $\CC$, which at least abstractly, means that their respective
coboundary operators are equally well understood. However, as soon as the conjugations
become wild, which is necessary for the goal of the construction, one cannot say many
things about the solutions of cohomological equations for diffeomorphisms in
$\TT \setminus \CC $, except, e.g., estimating their size as in eq. \ref{composition}.

Consequently, a lot of information is lost when one exits any precompact set in $\TT$ and
shoots for its boundary. Acknowledging this loss of information, one needs to accept that
a diffeomorphism $f \in \AKsm (M) \setminus \TT $ can be analyzed only via limit processes,
using the explicit properties of diffeomorphisms in $\CC$ and estimates on compositions.
This calls for the following definition.

\begin{definitionn} \label{def constructivity}
Suppose that a class $\CC \subset \Diffm$ is given, with a certain set of
properties $\Pi (\tilde{f})$ satisfied by $\tilde{f} \in \CC$, and consider the
corresponding Anosov-Katok space $\AKsm (\CC )$.

We will say that $f \in \AKsm (\CC )$ satisfies a certain property $\varpi$ in the constructive sence in
the context of the corresponding Anosov-Katok-like construction (and abridge to "$f $ has this certain
property constructively"), if the Anosov-Katok method can construct a sequence
$f_{n} \in \TT$, $f_{n} \ra f$ such that
\begeq
\forall n \in \N \, \exists \, \Pi (n) \subset \Pi , f_{n} \in \Pi _{n} \land f_{n} \overset{\sm }{\ra} f \Rightarrow
f \in \varpi
\endeq
\end{definitionn}

More informally, $f $ satisfies the property $\varpi$ in the constructive sense iff
said property for $f$ can be established through a limit procedure using only the known properties of
the class $\CC$ and the corresponding Anosov-Katok construction. An important remark is that
we do not demand that $\varpi \in \Pi$. For example, periodic rotations in the circle are not
uniquely ergodic, but irrational rotations are.% However, proving that certain rotations
%is $CR$ using only periodic approximation seems to be a difficult task.

The statement the "$f$ does
not satisfy constructively a certain property" means only that the known properties of the class
$\CC $ do not allow us to conclude whether $f$ has the property or not, at least in the context
of the construction.

Let us explain the definition through the example properties in which we are interested in the
present article. In the context of an Anosov-Katok-like construction, a diffeomorphism $f \in \AKsm (\CC )$
is constructively $DUE$ if the construction provides an approximating sequence $\{ f_{n} \} \subset \TT $
satisfying the following property. For every function $\f \in \smm$, for every
$s \in \N $ and for every $\e >0$, there
exists $n \in \N$ and a function $\f _{n} \in \smm$ such that
\begin{enumerate}
\item $\| \f - \f _{n} \| _{s} < \e $
\item There exists $\psi _{n} \in \smm$ such that
$\psi _{n} \circ f_{n} - \psi_{n} = \f _{n}$ (i.e. $\f _{n} \in Cob^{\infty}(f_{n})$)
\item The solution $\psi _{n}$ satisfies the following estimate (cf eq. \ref{composition 2})
\begin{equation} \label{eq error terms}
C_{s }\| \psi _{n} \| _{s+1} ||| f_{n} |||_{s} d_{s}(f,f_{n}) < \epsilon
\end{equation}
\end{enumerate}
This set of properties clearly implies that $f$ is $DUE$, since
\begeq
\psi _{n} \circ f - \psi _{n} = \f + (\f _{n} - \f ) + ( \psi _{n} \circ f - \psi _{n}\circ f_{n})
\endeq
and the terms in the parentheses sum up to $<2 \e $ in the $H^{s}$ norm by assumption. This
is the constructive interpretation of the approach adopted in \cite{AFKo2015}.
Since the class $\CC $ and the approximant sequence $f_{n}$ are constructed in some sense,
the solution $\psi _{n}$ should also be expected to be obtainable in a constructive way.
More precisely, we should expect the equation
\begeq
\tilde{\psi} \circ \tilde{f}_{n} - \tilde{\psi} = \tilde{\f}
\endeq
to be solvable with a good control for a sufficiently large space of functions $\tilde{\f}$,
and that this space should grow with $n$ and cover $\smm (M)$ in the limit $n \ra \infty$.
Then, the conjugation $H_{n}$ (through which $f_{n} = H_{n} \circ \tilde{f}_{n} \circ H_{n}^{-1}$
is defined), acting on the space of such functions $\tilde{\f }$ and on the solutions
$\tilde{\psi }$ by pullback should allow the three properties listed here above to be
established.

Inspection of the proof of proposition \ref{example prop rot vec 2} shows that the
diffomorphisms $f \in \AKsm (M)$ obtained by constructions for which, in the notation of
the proposition $\k \geq \t$, are \textit{not constructively} $DUE$. This is so, because
lemma \ref{lem closed graph} does not apply, which means that the approximation
$f_{n} \ra f$ is not fast enough for the solutions of equations
\begeq
\psi _{n} \circ f_{n} - \psi _{n} = \f
\endeq
to provide information for the (eventual) solutions of
\begeq
\psi \circ f - \psi = \f
\endeq
Since we admit the $f$ is analyzable only through the limit proceedure through which we
define it, we cannot say much about its coboundary operator.

Actually, proposition \ref{example prop rot vec 2} admits a constructive counterpart.
\begin{proposition} \label{example prop rot vec 2 const}
Let an Anosov-Katok-like construction satisfy standing assumptions \ref{ass inf dim ob}
to \ref{ass decay eta} and assumption \ref{ass rot vec case}.
Suppose, moreover, that only the upper bound of assumption \ref{ass cob op C} is
known to hold.

Then, the limit objects $f$ obtained by such constructions are not 
constructively $CR$.
\end{proposition}

\begin{proof}
The proof follows verbatim the one of proposition \ref{example prop rot vec 2}. Only
upper bounds are known, and they tend to infinty. In the case where $\k < \t$, this shows
that the construction cannot conclude whether $f$ is $CR $ or not.

In the case where $\k \geq \t$, lemma \ref{lem closed graph} does not apply, which means
that the coboundary operators of the approximants do not provide information on the
coboundary operator of $f$, at least not in the constructive sense.
\end{proof}

This proposition shows that, in the regime where the Anosov-Katok-like construction is
relevant for the construction of $CR$ examples, one of propositions
\ref{example prop comp gr}, \ref{example prop rot vec 1} or \ref{example prop rot vec 2}
applies. The boundary of this regime, loosely speaking, is given by the condition
$\k = \t $ of proposition \ref{example prop rot vec 2}, which makes the hypothesis
$\k < \t $ therein relevant and thus not restrictive.

The proof of cor. \ref{cor no counter-ex} can therefore be summarized as follows.
An Anosov-Katok-like
construction proceeds either by pushing all obstructions further to infinity at each step
("rotation vector case"), or by pushing only a part of them ("compact group case").
In the second case, the orders of magnitude are sufficient in order to conclude that
the $f \in \AKsm (\CC ) \setminus \TT $ for which the construction converges are not $CR$.
In the first case, we need an assumption on the solution of the cohomological equation
for diffeomorphisms in $\CC $ in order to conclude, but the assumptions are natural and
cover all known implementations.

We also make the following remark.

\begin{remark}
The assumptions at the $n$-th step that are stated as equalities have a tolerance of
$O(\eta _{s,m}^{-\infty})$, where $m$ is the correct scale, i.e. $m=n$ or $m=n \pm 1$.
\end{remark}

Even though we refrained from stating such sharp conditions in order to keep the argument
transparent, we invite the reader to verify for themselves this fact. This is in
accordance with the fact that if the construction were to suceed, it would be expected to
do so for a dense set in $\AKsm (\CC )$, in which case $DUE$ would be generic in the
Anosov-Katok space. This genericity is reflected in the tolerance of the conditions.

We conclude this section by remarking that the regime where proposition
\ref{example prop rot vec 2} applies is outside the classical generic regime of
$O(\r _{n}^{-\infty})$ conditions, where $\r _{n}$ is a fast growing sequence. The
proposition is relevant where there is some polynomial decay and touches the boundary
of the Anosov-Katok regime of $\AKsm (\CC )$.

\section{Relation with Almost Reducibility} \label{sec AR}

\subsection{Introduction} \label{sec intro cocycles}

The two cases treated in propositions \ref{example prop comp gr},
\ref{example prop rot vec 1} and \ref{example prop rot vec 2}
cover the Almost Reducibility regime (see \cite{El2001}, \cite{KrikAst}, \cite{NKPhD})
whenever it is obtained via a K.A.M. constructive procedure.
As a consequence, in spaces of Dynamical Systems where Almost
Reducibility can be established in a non highly exotic way, the Herman-Katok conjecture
should be expected to be true.

%The concept of almost reducibility is a quantitative version of the density of $\CC$ in
%$\AKsm (\CC )$ and asks that the size of conjugations not be big with respect to the
%mismatch. This is made explicit in the following definition.
%
%\begin{definitionn}
%In the context of an Anosov-Katok-like construction, we will say that the class $\CC$
%has the almost reducibility property in $\AKsm (\CC )$ iff for every $f \in \AKsm (\CC )$,
%there exist a sequence $\tilde{f}_{n} \in \CC $ and a sequence of conjugations $H_{n}$
%such that, if we call $f_{n } = H_{n} \circ \tilde{f}_{n} \circ H_{n}^{-1}$,
%\begeq
%d_{s}(f,f_{n}) \vertiii{H_{n}}_{s} \vertiii{H_{n}^{-1}}_{s} \ra 0, \forall s
%\endeq
%\end{definitionn}
%

Let us quickly explain the concept of Almost Reducibility, using the notation of \S
\ref{sec generalities & results}: we consider the
space of skew-product diffeomorphisms $SW^{\infty}_{\a}(\T ^{d} , P)$ over a fixed
rotation $\a \in \T ^{d}$.

%\begin{remark}
%The notation introduced in this section is local.
%\end{remark}
\begin{definitionn}
A cocycle is called constant iff $A(\. ) \equiv A \in G$ is a constant mapping, and
reducible iff it is smoothly conjugate to a constant one.
\end{definitionn}
Obviously, reducible cocycles
cannot be $CR$ unless $P \approx \T ^{d'}$ for some $d' \in \N ^{*}$. A theorem by H.
Eliasson, \cite{El2001}, establishes that even in the favourable case where
$\a \in DC$ and $A(\. )$ is a perturbation of a constant mapping in $SO(3)$, the cocycle
$(\a ,A(\. ))$ might not be (and in fact will generically not be) reducible. The author
showed in \cite{NKInvDist} that generically it will actually be
$DUE \setminus CR$ and never $CR$. Such a cocycle is nonetheless \textit{almost reducible}.
\begin{definitionn}
A cocycle is almost reducible iff there exists $G_{n}(\. )$, a sequence of mappings
$\T ^{d} \ra G$ and $A_{n}$ a sequence of constants in $G$ such that
\begeq
A_{n}^{-1} \left( G_{n}(\. +\a ). A(\. ). G_{n}^{-1}(\.) \right) \ra \Id
\text{ in } \sm (\T ^{d},G)
\endeq
\end{definitionn}
The sequence $B_{n}(\. )$ generically diverges, and generically it diverges precisely
because the cocycle $(\a , A(\. ))$ is not reducible (cf. \cite{NKRigidity},
\cite{NKInvDist}).

The reason for the divergence of the sequence $B_{n}(\. )$ is the phenomenon of
\textit{resonances}. Before defining the notion of resonances, let us briefly recall some
facts from the theory of compact Lie groups (see \cite{DieudElV} or \cite{NKPhD}).
\subsection{Facts from the theory of compact Lie groups} \label{sec facts Lie groups}
For each $A \in G$
there exists at least one torus $\TT \approx \T ^{d'}$, $\TT \hra G$, of maximal dimension $d'$ (depending
only on $G$) such that $A \in \TT$. Obviously, $\overline{ \{ A ^{k} \} _{k \in \Z} \} } \subseteq \TT$
for every such torus $\TT$. Given such a torus $\TT$, called a \textit{maximal torus}, we
can decompose the adjoint action of $G$ on $g =  T_{\Id}G$, the Lie algebra of $G$,
\begeq
Ad _{A}: g \ni s \mapsto \frac{d}{dt}A.e^{ts}.A^{-1} \upharpoonleft _{t =0}
\in g
\endeq
for every $A \in \TT$. By $\exp $ we denote the exponential mapping $g \ra G$ with
respect to the natural metric on $G$ given by the Cartan-Killing form
\begeq
\langle a, b \rangle = - \mathrm{tr} \left( s \mapsto [a,[b,s]] \right)
\endeq
and $[ \. , \. ]$ is the Lie bracket $ g\times g \ra g$. The commutator $[a,b]$ is given by
the derivative of the mapping $\R \ra G$, $t \mapsto e^{a}.e^{tb}.e^{-a}$ at
$t = 0$.\footnote{The cyclicity of the definitions is only apparent.}

The decomposition of the adjoint action into eigenspaces, known as root-space
decomposition, reads as follows.

Firstly, let us denote by $\ft \subset g$ the Lie algebra of $\TT$, i.e. $\ft = T _{\Id}\TT$.
It is a maximal abelian algebra, i.e. a maximal subspace of $g$ where the
Lie bracket vanishes identically. There exist pairwise orthogonal spaces
$\C j_{\r} \approx \C \approx \R ^{2} $, $\C j_{\r} \hra g$, which are orthogonal to $\ft$
and indexed by a finite set $\D \subset \ft ^{*} \setminus \{ 0 \}$ (called the roots of
$g$ with respect to $\ft$), satisfying the
following properties:
\begin{itemize}
\item for every $\r \in \D$, the vectors $j_{\r}$ and $i.j_{\r}$ ( $i \in \C$ is the
imaginary unit) are orthogonal.
\item for every $\r \in \D $, there exists a vector $h_{\r} \in \ft \setminus \{ 0\}$,
orthogonal to $\C j_{\r}$ and such that
\begeq
[h_{\r} , j_{\r} ] = i.j_{\r} 
\endeq
plus cyclic permutations. This can be summarized by saying that the vector space generated
by $\{h_{\r} , j_{\r} , i.j_{\r} \}$ defines an embedding of $su(2)$, the Lie algebra of
$SU(2)$ which is isomorphic to $\R ^{3}$ equipped with its scalar and vector product, into
$g$. This embedding is denoted by $(su(2))_{\r}$.
\item For $a \in \ft$ and $z \in \C$,
\begeq
Ad_{e^{a}}.(z.j_{\r}) = e^{2i\pi \r (a)}.z.j_{\r}
\endeq
\item There exists a subset $\D _{+} \subset \D = \D _{+} \cup (-\D _{+}) $ of roots such that every root $\r \in \D$
can be written in the form
\begeq
\r = \sum _{\r ' \in \D _{+}} m(\r , \r ') \r '
\endeq
with $m(\r , \r ' )$ integers of the same sign. The distinction with respect to the sign
is essentially the same as that between elements below and above the diagonal of a
unitary matrix.
\end{itemize}
Obviously, the eigenvalues of $Ad _{e^{a}}$ are $\{ e^{2i\pi \r (a)} \}_{\r \in \D} \cup
\{0\}$. We remark that $\r : \ft \ra \R$ take values in the real line, so that all eigenvalues
are in the unit circle. Root-space decompositions with respect to different maximal tori
are equivalent, since they are obtained by the adjoint action of the group onto itself.

In the context of the study of quasi-periodic cocycles, the following definition is very important.

\begin{definitionn} \label{def res coc}
Given $\a \in \T ^{d}$, an element $A$ of $G$ will be called \textit{resonant} with respect
to $\a$ iff there exists $\r \in \D$ and $k _{\r} \in \Z ^{d} \setminus \{ 0 \}$ such that
\begeq
\r (a) - k_{\r} \. \a \in \Z
\endeq
where $a \in \ft $ is any preimage of $A$ under the exponential mapping.

A constant cocycle $(\a ,A) \in SW^{\infty}_{\a}(\T ^{d} , G)$ is resonant iff $A \in G$ is
resonant with respect to $\a$, and the vector of integers $(k_{\r})$ (which is unique if
$\a$ is irrational) is called resonant mode.
\end{definitionn}

\subsection{The almost reducibility theorem}

Given the above, we can state the almost reducibility theorem referred to in \S
\ref{sec intro cocycles} as follows.
%We recall that if
%$P = G / H$, where $H $ is a closed subroup of $G$, the tangent space to the $\Id \in P$
%is $g / h$, where $g $, resp. $h$, are the Lie algebras of $G$, resp. $H$. The Lie algebra
%of a Lie group $G$ is the tangent space to the $\Id \in G$.
\begin{theorem} [\cite{El2001}, \cite{KrikAst}, \cite{NKPhD}] \label{thm almost red}
%Let $P = G/H$ be a homogeneous space of compact type, and $\a \in DC (\tilde{\gamma} , \tilde{\t } )$.
Let $G$ be a compact Lie group, and $\a \in DC (\tilde{\gamma} , \tilde{\t } )$. Then,
there exist $s_{0} >0$ and $\e >0 $ such that if the cocycle $(\a , Ae^{F(\. )})$,
$A \in G $ and $F(\. ) \in \sm (\T ^{d} , g ) $ satisfies
\begeq
\| F(\. ) \|_{0} < \e \text{ and } \| F(\. ) \|_{s_{0}} < 1
\endeq
then it is Almost Reducible.

More precisely, the K.A.M. scheme that proves Almost Reducibility produces:
\begin{enumerate}
\item a fast increasing sequence $N_{n} \in \N ^{*}$, $N_{n+1} = N_{n}^{1+ \d }$ for some
$0 < \d < 1$
\item a sequence of constants $A _{n} \in G$
\item a number $\t > \tilde{\t }$, a subsequence $n_{k}$ and a set of resonant constants
$\{ \L _{k}  \} _{k =1}^{M} \subset G$ whose resonance $k_{n_{k}}^{\r} \in \Z ^{*}$ in each $(su(2))_{\r}$,
if it is defined, satisfies $|k_{n_{k}}^{\r} | \leq N_{n_{k}}$, as well as the following properties.
The constant $\L _{k}$ commutes with $A_{n_{k}}$ and $d(\L _{k}  , A_{n_{k}} ) <
N_{n_{k}} ^{-\t }$, if such an $\L _{n_{k}}$ exists.
% Moreover, there exists a subalgebgra
%$(su(2)) _{k} \approx su(2) $ of $g $ such that
%\begeq
%d(Ad(\L _{k})\upharpoonright _{(su(2)) _{k}} , \{e \} ) < N_{n} ^{-\t }
%\endeq
The number
$M \in \N \cup \{ \infty \} $ counts the number of resonant steps of the K.A.M. scheme, i.e. those for which
$\L _{n_{k}}$ is defined
\item \label{item exp dec Y} a sequence of conjugations $Y_{n}(\. ) \in \sm (\T ^{d} ,g)$ satisfying
\begeq
\| Y_{n}(\. ) \|_{s} = O( N_{n}^{-\infty }), \forall s \geq 0
\endeq
\item \label{item pol growth B} a sequence of conjugations $B_{n_{k}}(\. ) \in \sm (\T ^{d} ,G)$ satisfying
\begeq
\| B_{n_{k}}(\. ) \|_{s} \simeq C_{s} N_{n_{k}}^{s+ \l }, \forall s \geq 0
\endeq
for some constant $\l >0$. These conjugations commute with the respective $A_{n_{k}}$ and
$\L _{n_{k}}$, and the constant
\begeq
B_{n_{k}}(\. + \a ). \L _{n_{k}} .B_{n_{k}}^{-1}(\. )
\endeq
is $N_{n_{k}}^{-\t }$-away from resonant constants. If $A_{n} $ is $N_{n}^{-\t }$-away
from resonant constants (i.e. if $n$ is not a resonant step) then $B_{n}(\. )$ is by
convention defined as $\equiv \Id$
\item a sequence of mappings $F_{n}(\. ) \in \sm (\T ^{d}, g)$ satisfying
\begeq
\| F_{n}(\. ) \|_{s} = O( N_{n}^{- \infty }), \forall s \geq 0
\endeq
\end{enumerate}
and such that the conjugation constructed iteratively following $G_{0} = \Id$ and
\begeq
G_{n} (\. )= \begin{cases}
e^{Y_{n}(\. )} G_{n-1} (\. ) \text{ if } n \notin \{ n_{k} \} \\
B_{n}(\. ) e^{Y_{n}(\. )} G_{n-1}(\. ) \text{ if } n \in \{ n_{k} \}
\end{cases}
\endeq
satisfies
\begeq
G_{n} (\. + \a ) Ae^{F(\. )} G_{n}^{-1}(\. ) = A_{n}. e^{F_{n}(\. )}
\endeq
\end{theorem}

\begin{remark}
The parameter $\t$ in the statement of the theorem corresponds to the parameter $\t$ in
assumption \ref{ass decay eta}.
\end{remark}

As we have already pointed out in \cite{NKInvDist} and \cite{NKContSpec}, the close-to-the-$\Id$
conjugations $Y_{n}$ are highly redundant. If we rearrange them with the $B_{n_{k}}$ following
\begeq
B_{n_{2}}e^{Y_{n_{2}}}\cdots e^{Y_{n_{1}+1}} B_{n_{1}}e^{Y_{n_{1}}}\cdots e^{Y_{1}}=
B_{n_{2}}B_{n_{1}}e^{\tilde{Y}_{n_{2}}}\cdots e^{\tilde{Y}_{n_{1}+1}} e^{Y_{n_{1}}}\cdots e^{Y_{1}}
\endeq
and similarly for the rest of the steps, the product formed by the $e^{\tilde{Y}_{n}}$
converges thanks to items \ref{item exp dec Y} and \ref{item pol growth B} of thm.
\ref{thm almost red}. If we call the resulting product $D (\. )$, then the cocycle
\begeq
\tilde{A } (\. ) = D(\. + \a) Ae^{F(\. )}D^{-1}(\. )
\endeq
is in what we called \textit{K.A.M. normal form} in the references. This means that, up to
a second order perturbation, the cocycle
\begeq
(B_{n_{k}} \cdots B_{n_{1}})(\. + \a )\tilde{A}(\. ) ((B_{n_{k}} \cdots B_{n_{1}}))^{-1}(\. ) = \tilde{A}_{k}e^{\tilde{F}_{k}(\. )}
\endeq
is either constant or has the following particular structure. Consider a root-space with
respect to a torus passing by $\tilde{A}_{k}$. If the root $\r $ is not resonant for
$\L _{k}$, then the restriction of $\tilde{F}_{k}(\. )$ in $(su(2))_{\r}$ is a constant. The
constant is in $\ft$ if the corresponding eigenvalue of $\L_{n} $ is not equal to $1$, and the constant is
in $(su(2))_{\r}$ if the eigenvalue is equal to $1$. If $\r $ is resonant,
for the resonance $k_{\r}$, the restriction is a constant in $\ft $ plus
$\hat{\tilde{F}}(k_{\r})e^{2i\pi k_{\r}\.}j_{\r}$. In short, to the first order only
the resonant modes are active, and this particular structure allows us to accurately estimate
the commutativity (or lack thereof) of the constants $\tilde{A}_{n_{k}}$ and $\tilde{A}_{n_{k}+1}$. Informally,
if the commutator is significantly away from the $\Id$ for an infinite number of resonant steps,
then the dynamics will be weakly mixing in the fibers (c.f. \cite{NKContSpec}). If the commutators
visit a certain small set depending only on $G$ infinitely often, the dynamics will be $DUE$ (c.f. \cite{NKInvDist}
for the case $G = SU(2)$).

In what follows, we assume the cocycle in normal form and
ommit the tilde in the notation, while keeping the rest of the notation the same.

\subsection{The study of the cohomological equation}

The study of the invariant distributions of almost reducible cocycles was possible because
of the relation between Almost Reducibility and the Anosov-Katok construction. The former,
obtained by an application of K.A.M. theory proves "almost rigidity" for perturbations of
constant cocycles (i.e. perturbations of constant cocycles are almost conjugate to
constant ones) and gives very good control on the failure of rigidity. When rigidity
fails, i.e. when a perturbation of a constant cocycle is not conjugate to a constant one,
this control allows almost reducibility to be interpreted as a fast approximation by
conjugation scheme for the given cocycle. The approximant cocycles are resonant (their
coboundary space is smaller than that of generic constant ones), and the estimates furnished
by the K.A.M. scheme make analysis extremely efficient.

In what follows, we adapt notation from the references to notation of the present work.
In \cite{NKRigidity}, we proved, for cocycles in $\T ^{d} \times SU(2)$,
that when prop. \ref{prop approx} is not applicable because the rate of convergence
is polynomial with respect to the corresponding sequence $\l _{N_{n}}$, the cocycle
is $\sm$ reducible. In \cite{NKInvDist} we showed that when the rate is exponential
(and it is so for a generic Almost Reducible cocycle), the cocycle is generically
$DUE$. In the same work we showed that such dynamical systems are never
$CR$, and the proof of cor. \ref{cor no counter-ex} is an abstraction and a generalization
of that proof.

The connection with the Anosov-Katok construction is established via the following
dictionary. If we take $\CC$ to be the class of constant, diagonal, resonant cocycles,
then, by minimality of $\a$, $\bar{\CC}$ is the class of diagonal cocycles. By definition,
$\TT$ is dense in the class of reducible cocycles. The almost reducibility theorem then
implies that $\AKsm (\CC)$ contains the open set that we call the KAM regime.

In fact, the following theorem is well in the reach of the techniques developed in
\cite{NKRigidity} and \cite{NKInvDist}, but without the tools developed in the present
article, the proof would be unnecessarily involved.
\begin{theorem} \label{thm no counter ex in KAM regime}
Let $d \in \N ^{*}$ and $G$ be a compact group. Then, in the K.A.M. regime of
$SW^{\infty }(\T ^{d} , P)$, there are no counter-examples to conjecture \ref{conj CR}.
\end{theorem}

\begin{proof}
 The K.A.M. regime is the set of cocycles $(\a ,A(\. ))$ to which thm. \ref{thm almost red}
applies. We assume the cocycle in K.A.M. normal form.

%For such cocycles, a converging K.A.M. scheme can be constructed as in \cite{El2001},
%proving almost reducibility with precise estimates. The scheme produces a sequence
%$A_{n} \in G$ and a sequence $F_{n} \in C^{\infty}(\T ^{d} ,g)$ and a sequence
%$N_{n} \in \N$, $N_{n} \ra \infty$ fast.
When the space $\T ^{d} \times P$ is equipped
with its natural Riemannian structure, almost reducibility, interpreted as in
\cite{NKInvDist} or \cite{NKContSpec}, produces a sequence of cocycles $(\a ,A_{n}(\. ))$
converging to $(\a ,A(\. ))$ which, together with the sequence $N_{n}$ satisfy the
conditions of prop. \ref{prop approx} for any $\s \geq 0$, except possibly for the fast approximation
condition $\d _{\s ,n} = O(\l _{N_{n}}^{-\infty})$.

Thus, two cases can occur.
\begin{enumerate}
\item Either the fast approximation condition is satisfied, something which translates to
\begeq
d (A_{n_{k}} , \L _{k}) = O(N_{n_{k}}^{- \infty})
\endeq
and $(\a ,A(\. )) \notin CR$ by prop \ref{example prop rot vec 1}.
\item Or there exist $\gamma ',\t ' >0$ such that
\begeq
d (A_{n_{k}} , \L _{k}) \geq \gamma ' N_{n_{k}}^{- \t '}
\endeq
In this case, an adjustment of the parameters of the scheme, as in \cite{NKRigidity} can
show that the cocycle is actually reducible, which results in the phase space foliating in
invariant tori.
\end{enumerate}
The proof is complete.
\end{proof}

%
%The assumptions in the proof of cor. \ref{cor no counter-ex}, and more precisely of prop.
%\ref{example prop comp gr}, cover the cases where almost
%reducibility is obtained via a K.A.M. scheme in any compact group $G$, and therefore one
%should not expect to find a counter example in any space of dynamical systems where almost
%reducibility holds in the K.A.M. constructive sense. In that case, one either has
%\begeq
%\d _{s,n} = O(\l _{N_{n}}^{-\infty}), \forall s \in \N
%\endeq
%or, in the case of at most polynomial decay, some kind of rigidity phenomenon appears,
%cf. \cite{NKRigidity}, and an invariant foliation arises as an obstruction to $DUE$.

If one works a bit harder and generalizes the proof of genericity of $DUE$ in the KAM regime
of $SW^{\infty}_{\a}(\T ^{d},SU(2))$, where $\a \in CD$, they can show that non-reducible
cocycles in the KAM regime of $SW^{\infty}_{\a}(\T ^{d},P)$ satisfy the hypothesis of both
propositions \ref{example prop comp gr} and \ref{example prop rot vec 1}.

For the former, it is so because the functional space of each representation is invariant.
If one considers a sequence of approximating cocycles $(\a ,A_{n}(\. ))$, then the following
holds. If one imposes a condition guaranteeing that assumption \ref{ass rot vec case} holds
in one irreducible representation $\pi _{1}$, then there will exist another representation
$\pi_{2}$ for which assumption \ref{ass comp gr case} holds. Moreover, one can choose the
difference of the dimensions of $\pi _{1}$ and $\pi _{2}$ to be bounded, so that all
assumptions of proposition \ref{example prop comp gr} be satisfied.

For the latter, one can actually show that, if one restricts to the functional space
corresponding to a single irreducible representation, assumption \ref{ass comp gr case} is
still satisfied. This is so, because for all admissible conjugations, measures preserved
by the cocycles $(\a ,B_{n}(\. +\a ) A_{n}(\. ) B_{n}(\. )^{-1} )$ and $(\a ,A_{n+1}(\. ))$
have an intersection whose dimension is equal to the dimension of the representation
(see lemma 7.9 of \cite{NKInvDist}, item 1 of the conclusions for the case $P=SU(2)$).
Consequently, even in each invariant function space, proposition \ref{example prop comp gr}
applies. Cohomological rigidity in such spaces seems, therefore, hopeless, as multiple
obstructions appear.

We remark that proposition \ref{example prop rot vec 2} is not needed here, since the
theorem of differentiable rigidity (\cite{NKRigidity}) kicks in when the rate of decay of
$\eta _{s,n}$ is polynomial with respect to the resonances.

The other important model of cocycles, apart from constant ones, is given by the periodic
geodesics of the group $G$ (see \cite{Krik2001} or \cite{NKPhD}, chapters $4$ and $8$). They do not constitute
a good basis for an $\AKsm$ space, though, for the following reason. They are modeled upon
the parabolic map
\begeq
\T \times \T \ni (x,y) \mapsto (x+ \a , y + r x) \in \T \times \T
\endeq
for some $r \in \N^{*}$, instead of a the quasi-periodic mapping
\begeq
\T \times \T \ni (x,y) \mapsto (x+ \a , y + \beta ) \in \T \times \T
\endeq
modelling the constant ones. Invariant distributions for the parabolic map can be calculated by hand, or see \cite{Kat01}. The
calculation shows that, unless one allows $r \ra \infty$, the assumptions of lem.
\ref{lem approx dist comp} are satisfied and no $DUE$ example can be constructed in the
corresponding space.

As long as Almost Reducible cocycles and cocycles that can be conjugated arbitrarily close
to periodic geodesics of $G$ fill $SW_{\a}^{\infty}(\T ,G)$ for some $\a$, then no counter-examples
to conj. \ref{conj CR} exist in that space. Theorem $1.3$ in \cite{NKPhD} argues that
this is the case when $\a \in RDC$, which proves corollary \ref{cor no counter-ex in SW}.
%We actually think that thm $1.3$ of the reference should hold under a classical Diophantine
%Condition, which would make cor. \ref{cor no counter-ex in SW} true in
%$SW^{\infty}(\T ,G)$.
%
%\begin{question} %[to Flaminio \& Forni]
%Would the same phenomenon be present in cocycles with values in the affine group (cf. 
%\cite{FlFoRH16})?
%\end{question}

\section{Conclusions and comments}

The proof of corollary \ref{cor no counter-ex} shows that, if one wishes to
construct a counter-example to the Herman-Katok conjecture, they have indeed a very
difficult task to accomplish, since using the most powerful method for constructing
realizations of non-standard dynamics in the elliptic case appears to be a bad strategy. As
soon as approximation is fast, cf. def. \ref{def fast approx}, they need to be able
to treat all obstructions arising at each step in the immediately next one, and they should
be able to do so with estimates that seem to be out of reach for the existing arsenal of
examples. On the other hand, slow approximation (i.e. at a polynomial rate) seems to
result in the persistence of some structure obstructing $DUE$.

The whole approach of the article comes from intuition built on elliptic dynamics, and
especially quasi-periodic dynamics. This context is precisely the origin of the Katok-Herman
conjecture, which informally states that K.A.M. theory is perturbation theory for
rotations in tori, and that only they can serve as its linear model.

Our approach seems to be disjoint from those in the litterature. For example, the proof of
the conjecture for flows in dimension $3$ (\cite{Forn08}, \cite{Koc09}, \cite{Mats09},
\cite{RHRH06}),
the first non-trivial case for the continuous-time version of the conjecture, is based on
techniques and results from dynamical systems and differential topology, and some very
heavy machinery from symplectic topology. This symplectic topology machinery, namely the
Weinstein conjecture, is used in order to exclude the case where the vector field is the
Reeb vector field of a contact form in a cohomological sphere, in which case $CR$ fails
quite dramatically, due to the existence of periodic orbits.

The study of the cohomological equation for circle diffeomoprhisms in \cite{AKo11} uses
the renormalization scheme and depends very heavily on the existence of an order in
$\T ^{1}$ which has been systematically exploited throughout the developement of the
theory (cf. \cite{Herm79}, \cite{YocAst}).

The study of the cohomological equation for homogeneous flows (\cite{FlFoRH13}, \cite{FlaFor07}), is
naturally based on representation theory, as was the article by the author, \cite{NKInvDist}.
The works by L. Flaminio, G. Forni and F. Rodriguez-Hertz go further and deeper than the
verification of conj. \ref{conj CR}, but as far as the conjecture itself is concerned,
\textit{a posteriori} they seem to be more or less the end of the road. This is so, because
such flows or diffeomorphisms are found to always have an infinite codimensional space of
coboundaries, and, consequently, in those classes $CR$ fails quite dramatically. Of course,
the infinite codimensionality of the coboundary space is precisely the object of the proof
and it is \textit{a priori} not at all obvious.
%
%\begin{question} [Probably naive]
%How are \cite{FlFoRH13} and \cite{FlaFor07} related with periodic approximation? Could
%non-trivial distributions be studied through approximation by periodic flows? Do periodic
%flows or affine mappings form a dense subset?
%\end{question}

On the other hand, in the space of cocycles in $\T ^{d} \times SU(2)$, the space studied
by the author in \cite{NKInvDist}, $DUE$ is a generic property, and the proof of the 
non-existence of
$CR$ examples can give some insight into the mechanism that both creates $DUE$ and
obstructs $CR$. This insight is precisely what led to the present article.

Despite the fact that harmonic analysis on manifolds is less efficient and
less elegant than representation theory for homogeneous spaces, the present article
suggests that its use can lead to advances in the study of the conjecture.

\bibliography{aomsample}

\providecommand{\bysame}{\leavevmode\hbox to3em{\hrulefill}\thinspace}
\providecommand{\noopsort}[1]{}
\providecommand{\mr}[1]{\href{http://www.ams.org/mathscinet-getitem?mr=#1}{MR~#1}}
\providecommand{\zbl}[1]{\href{http://www.zentralblatt-math.org/zmath/en/search/?q=an:#1}{Zbl~#1}}
\providecommand{\jfm}[1]{\href{http://www.emis.de/cgi-bin/JFM-item?#1}{JFM~#1}}
\providecommand{\arxiv}[1]{\href{http://www.arxiv.org/abs/#1}{arXiv~#1}}
\providecommand{\doi}[1]{\url{http://dx.doi.org/#1}}
\providecommand{\MR}{\relax\ifhmode\unskip\space\fi MR }
\providecommand{\MRhref}[2]{%
  \href{http://www.ams.org/mathscinet-getitem?mr=#1}{#2}
}
\providecommand{\href}[2]{#2}
\begin{thebibliography}{PAIfASP65}

\bibitem[AK70]{AnKat1970}
\bgroup\scshape{}D.~V. Anosov\egroup{} and \bgroup\scshape{}A.~B.
  Katok\egroup{}, New examples in smooth ergodic theory. {E}rgodic
  diffeomorphisms,  \emph{Trudy Moskov. Mat. Ob\v s\v c.} \textbf{23} (1970),
  3--36. \mr{0370662 (51 \#6888)}.

\bibitem[AFK15]{AFKo2015}
\bgroup\scshape{}A.~Avila\egroup{}, \bgroup\scshape{}B.~Fayad\egroup{}, and
  \bgroup\scshape{}A.~Kocsard\egroup{}, On manifolds supporting
  distributionally uniquely ergodic diffeomorphisms,  \emph{J. Differential
  Geom.} \textbf{99} (2015), 191--213.

\bibitem[AFK11]{AFK2011}
\bgroup\scshape{}A.~Avila\egroup{}, \bgroup\scshape{}B.~Fayad\egroup{}, and
  \bgroup\scshape{}R.~Krikorian\egroup{}, A {KAM} scheme for {${\rm SL}(2,\Bbb
  R)$} cocycles with {L}iouvillean frequencies,  \emph{Geom. Funct. Anal.}
  \textbf{21} (2011), 1001--1019. \mr{2846380}.

\bibitem[AK11]{AKo11}
\bgroup\scshape{}A.~Avila\egroup{} and \bgroup\scshape{}A.~Kocsard\egroup{},
  Cohomological equations and invariant distributions for minimal circle
  diffeomorphisms,  \emph{Duke Math. J.} \textbf{158} (2011), 501--536.
  \mr{2805066 (2012f:37087)}.  \doi{10.1215/00127094-1345662}.

\bibitem[Die75]{DieudElV}
\bgroup\scshape{}J.~Dieudonn\'{e}\egroup{}, \emph{El\'{e}ments d'Analyse, 5},
  Gauthier-Villars, 1975.

\bibitem[Eli01]{El2001}
\bgroup\scshape{}L.~H. Eliasson\egroup{}, Almost reducibility of linear
  quasiperiodic systems,  \emph{Proceedings of Symposia in Pure Mathematics}
  \textbf{69} (2001).

\bibitem[FK04]{FayadKatok2004}
\bgroup\scshape{}B.~Fayad\egroup{} and \bgroup\scshape{}A.~Katok\egroup{},
  Constructions in elliptic dynamics,  \emph{Ergodic Theory and Dynamical
  Systems} \textbf{24} (2004), 1477--1520.

\bibitem[FF03]{FlaFo03}
\bgroup\scshape{}L.~Flaminio\egroup{} and \bgroup\scshape{}G.~Forni\egroup{},
  Invariant distributions and time averages for horocycle flows,  \emph{Duke
  Math. J.} \textbf{119} (2003), 465--526. \mr{2003124}.
  \doi{10.1215/S0012-7094-03-11932-8}.

\bibitem[FF07]{FlaFor07}
\bgroup\scshape{}L.~Flaminio\egroup{} and \bgroup\scshape{}G.~Forni\egroup{},
  On the cohomological equation for nilflows,  \emph{J. Mod. Dyn.} \textbf{1}
  (2007), 37--60. \mr{2261071}.

\bibitem[FFRH13]{FlFoRH13}
\bgroup\scshape{}L.~Flaminio\egroup{}, \bgroup\scshape{}G.~Forni\egroup{}, and
  \bgroup\scshape{}F.~Rodriguez~Hertz\egroup{}, \emph{Invariant Distributions
  for homogeneous flows}, 2013. \arxiv{1303.7074}.

\bibitem[For08]{Forn08}
\bgroup\scshape{}G.~Forni\egroup{}, On the {G}reenfield-{W}allach and {K}atok
  conjectures in dimension three,  in \emph{Geometric and probabilistic
  structures in dynamics}, \emph{Contemp. Math.} \textbf{469}, Amer. Math.
  Soc., Providence, RI, 2008, pp.~197--213. \mr{2478471}.
  \doi{10.1090/conm/469/09167}.

\bibitem[Her79]{Herm79}
\bgroup\scshape{}M.~R. Herman\egroup{}, Sur la conjugaison diff\'erentiable des
  diff\'eomorphismes du cercle \`a des rotations,  \emph{Inst. Hautes \'Etudes
  Sci. Publ. Math.} (1979), 5--233. \mr{538680 (81h:58039)}.

\bibitem[Her80]{Herm80}
\bgroup\scshape{}M.-R. Herman\egroup{}, R\'esultats r\'ecents sur la
  conjugaison diff\'erentiable,  in \emph{Proceedings of the {I}nternational
  {C}ongress of {M}athematicians ({H}elsinki, 1978)}, Acad. Sci. Fennica,
  Helsinki, 1980, pp.~811--820. \mr{562693}.

\bibitem[Hur85]{HurdKatConj}
\bgroup\scshape{}S.~Hurder\egroup{}, Problems on rigidity of group actions and
  cocycles,  \emph{Ergodic Theory Dynam. Systems} \textbf{5} (1985), 473--484.
  \mr{805843 (87a:58098)}.

\bibitem[Kar14]{NKInvDist}
\bgroup\scshape{}N.~Karaliolios\egroup{}, \emph{Invariant distributions for
  quasiperiodic cocycles in $\T ^{d} \times SU(2)$}, 2014. \arxiv{1407.4763}.

\bibitem[Kar15]{NKContSpec}
\bgroup\scshape{}N.~Karaliolios\egroup{}, \emph{Continuous spectrum or
  measurable reducibility for quasiperiodic cocycles in $\mathbb{T} ^{d} \times
  SU(2)$}, 2015. \arxiv{1512.00057}.

\bibitem[Kar16a]{NKKAMTor}
\bgroup\scshape{}N.~Karaliolios\egroup{}, \emph{Local Rigidity of Diophantine
  translations in higher dimensional tori}, 2016. \arxiv{1612.05564}.

\bibitem[Kar16b]{NKPhD}
\bgroup\scshape{}N.~Karaliolios\egroup{}, Global aspects of the reducibility of
  quasiperiodic cocycles in semisimple compact {L}ie groups,  \emph{M\'em. Soc.
  Math. Fr. (N.S.)} (2016), 4+ii+200. \mr{3524104}.

\bibitem[Kar17]{NKRigidity}
\bgroup\scshape{}N.~Karaliolios\egroup{}, Differentiable rigidity for
  quasiperiodic cocycles in compact {L}ie groups,  \emph{J. Mod. Dyn.}
  \textbf{11} (2017), 125--142. \mr{3627120}.

\bibitem[Kat01]{Kat01}
\bgroup\scshape{}A.~Katok\egroup{}, Cocycles, cohomology and combinatorial
  constructions in ergodic theory,  in \emph{Smooth ergodic theory and its
  applications ({S}eattle, {WA}, 1999)}, \emph{Proc. Sympos. Pure Math.}
  \textbf{69}, Amer. Math. Soc., Providence, RI, 2001, In collaboration with E.
  A. Robinson, Jr., pp.~107--173. \mr{1858535 (2003a:37010)}.

\bibitem[KH96]{KatHass}
\bgroup\scshape{}A.~Katok\egroup{} and
  \bgroup\scshape{}B.~Hasselblatt\egroup{}, \emph{Introduction to the Modern
  Theory of Dynamical Systems}, Cambridge University Press, 1996.

\bibitem[Koc09]{Koc09}
\bgroup\scshape{}A.~Kocsard\egroup{}, Cohomologically rigid vector fields: the
  {K}atok conjecture in dimension 3,  \emph{Ann. Inst. H. Poincar\'e Anal. Non
  Lin\'eaire} \textbf{26} (2009), 1165--1182. \mr{2542719 (2010h:37053)}.

\bibitem[Kri99]{KrikAst}
\bgroup\scshape{}R.~Krikorian\egroup{}, R\'eductibilit\'e des syst\`emes
  produits-crois\'es \`a valeurs dans des groupes compacts,
  \emph{Ast\'erisque} (1999), vi+216. \mr{1732061 (2001f:37030)}.

\bibitem[Kri01]{Krik2001}
\bgroup\scshape{}R.~Krikorian\egroup{}, Global density of reducible
  quasi-periodic cocycles on ${T}^1 \times {SU}(2)$,  \emph{Annals of
  Mathematics} \textbf{154} (2001), 269--326.

\bibitem[Kuk00]{Kuksin2000AnHamPDES}
\bgroup\scshape{}S.~Kuksin\egroup{}, \emph{Analysis of Hamiltonian PDEs},
  Oxford University Press, 2000.

\bibitem[Mat09]{Mats09}
\bgroup\scshape{}S.~Matsumoto\egroup{}, The parameter rigid flows on orientable
  3-manifolds,  in \emph{Foliations, geometry, and topology}, \emph{Contemp.
  Math.} \textbf{498}, Amer. Math. Soc., Providence, RI, 2009, pp.~135--139.
  \mr{2664594}.  \doi{10.1090/conm/498/09746}.

\bibitem[PAIfASP65]{PalaisAtiyahIndThm}
\bgroup\scshape{}R.~Palais\egroup{}, \bgroup\scshape{}M.~Atiyah\egroup{}, and
  \bgroup\scshape{}N.~Institute~for Advanced Study~(Princeton\egroup{},
  \emph{Seminar on the Atiyah-Singer Index Theorem}, \emph{Annals of
  mathematics studies}, Princeton University Press, 1965.

\bibitem[RHRH06]{RHRH06}
\bgroup\scshape{}F.~Rodriguez~Hertz\egroup{} and
  \bgroup\scshape{}J.~Rodriguez~Hertz\egroup{}, Cohomology free systems and the
  first {B}etti number,  \emph{Discrete Contin. Dyn. Syst.} \textbf{15} (2006),
  193--196. \mr{2191392}.  \doi{10.3934/dcds.2006.15.193}.

\bibitem[R\"02]{Russ02}
\bgroup\scshape{}H.~R\"ussmann\egroup{}, Stability of elliptic fixed points of
  analytic area-preserving mappings under the {B}runo condition,  \emph{Ergodic
  Theory Dynam. Systems} \textbf{22} (2002), 1551--1573. \mr{1934150}.
  \doi{10.1017/S0143385702000974}.

\bibitem[Yoc95]{YocAst}
\bgroup\scshape{}J.-C. Yoccoz\egroup{}, \emph{Petits diviseurs en dimension 1,
  Ast\'{e}risque 231}, Soci\'{e}t\'{e} Math\'{e}matique de France, 1995.

\end{thebibliography}
\bibliographystyle{aomalpha}

\end{document}